\documentclass[review]{elsarticle}

\usepackage{lineno,hyperref}
\modulolinenumbers[5]

\journal{\textcolor{white}{a}}

\usepackage[utf8]{inputenc}
\usepackage[margin=2.5cm]{geometry}
\usepackage{amsmath,bm,bbm,amsthm, amssymb}
\usepackage{hyperref}
\usepackage[dvipsnames]{xcolor}
\usepackage{soul}
\usepackage{mathtools}
\usepackage{dsfont}

\usepackage{tikz}
\usepackage{pgfplots}

\hypersetup{
    colorlinks=true,
    linkcolor=blue,
    filecolor=magenta,
    urlcolor=Blue,
    citecolor=Blue,
}

\pgfplotsset{compat=1.17} 
\newtheorem{theorem}{Theorem}[section]
\newtheorem{lemma}[theorem]{Lemma}
\newtheorem{proposition}[theorem]{Proposition}
\newtheorem{corollary}[theorem]{Corollary}
\newtheorem{remark}[theorem]{Remark}

\theoremstyle{definition}
\newtheorem{example}[theorem]{Example}

\newcommand{\norm}[1]{|#1|}
\def\Comp{\mathsf{Comp}}

\newcommand{\cb}{\bs c}
\newcommand{\Rb}{\bs G}

\newcommand{\aaa}{\operatorname{\mathfrak{a}}}

\def\mb{\mathbb}
\def\bs{\boldsymbol}
\def\mc{\mathcal}

\newcommand{\E}{\operatorname{\mb E}}
\newcommand{\R}{\mb R}
\newcommand{\N}{\mb N}

\newcommand{\Z}{\mb Z}
\newcommand{\s}{\mb S}

\renewcommand{\P}{\operatorname{\mb P}}

\def\tff{\uparrow \infty}

\def\la{\lambda}

\def\b{\beta}
\def\bet{\begin{theorem}}
\def\ent{\end{theorem}}
\def\f{\frac}
\def\one{\operatorname{\mathds{1}}}%{\mathbbmss{1}}

\def\bel{\begin{lemma}}
\def\enl{\end{lemma}}
\def\bep{\begin{proof}}
\def\enp{\end{proof}}
\def\Var{\operatorname{\ms{Var}}}
\def\ms{\mathsf}
\def\r{\rho}
\def\su{\subseteq}
\def\e{\varepsilon}
\def\bepr{\begin{proposition}}
\def\enpr{\end{proposition}}
\def\ot{[0, 1]}
\def\ff{\infty}

\def\s{\sigma}
\def\GG{\mc G}
\def\PP{\mc P}
\def\RR{\mc G_{d}}
\def\RRk{\mc G_{d,k}}
\def\AA{\mc A}

\def\Cov{\operatorname{\ms{Cov}}}
\def\dist{\operatorname{\ms{dist}}}

\def\been{\begin{enumerate}}
\def\enen{\end{enumerate}}
\def\im{\item}
\def\diam{\operatorname{\ms{diam}}}

\def\co{\colon}
\def\sm{\setminus}

\def\De{\Delta}

\def\d{\,\mathrm{d}}

\def\xx{{\boldsymbol x}}
\def\yy{{\boldsymbol y}}
\def\zz{{\boldsymbol z}}

\def\vp{\varphi}
\def\es{\emptyset}
\def\bs{\boldsymbol}

\def\lan{\langle}

\def\ran{\rangle}
\def\a{\alpha}
\def\dxx{\d\xx}
\def\cmm{c_{\ms{MM}}}

\definecolor{wiasblue}   {cmyk}{1.0, 0.60, 0, 0}

\def\CCC{\bs G}
\def\tb{\bs t}

\def\Rips{\ms{GG}}
\def\II{\mc A^{(m)}}
\def\cdo{c_{\ms{Dom}}}

\def\bee{\begin{example}}
\def\ene{\end{example}}
\def\vm{v_{\ms{max}}}
\def\bec{\begin{corollary}}
\def\enc{\end{corollary}}
\def\coloneqq{:=}
\def\subsetneq{\varsubsetneq}
\begin{document}

\begin{frontmatter}

    \title{Limit theory of sparse random geometric graphs in high dimensions}
    
    \author[inst1,inst2]{Gilles Bonnet}
    \author[inst3,inst4]{Christian Hirsch}
    \author[inst5]{Daniel Rosen}
    \author[inst1,inst2]{Daniel Willhalm\corref{cor1}} \ead{d.willhalm@rug.nl}

    \cortext[cor1]{Corresponding author}
    
    \affiliation[inst1]{organization={Bernoulli Institute, University of Groningen},
                addressline={Nijenborgh 9}, 
                city={Groningen},
                postcode={9747 AG}, 
                country={Netherlands}}
    
    \affiliation[inst2]{organization={CogniGron (Groningen Cognitive Systems and Materials Center)},
                addressline={Nijenborgh 4}, 
                city={Groningen},
                postcode={9747 AG}, 
                country={Netherlands}}
                
    \affiliation[inst3]{organization={Department of Mathematics, Aarhus University},
                addressline={Ny Munkegade 118}, 
                city={Aarhus C},
                postcode={8000}, 
                country={Denmark}}
    
    \affiliation[inst4]{organization={DIGIT Center, Aarhus University},
                addressline={Finlandsgade 22}, 
                city={Aarhus N},
                postcode={8200}, 
                country={Denmark}}
    
    \affiliation[inst5]{organization={Faculty of Mathematics, Ruhr University Bochum},
                city={Bochum},
                postcode={44780}, 
                country={Germany}}

    \begin{abstract}
        We study topological and geometric functionals of $l_\infty$-random geometric graphs on the high-dimensional torus in a sparse regime, where the expected number of neighbors decays exponentially in the dimension. More precisely, we establish moment asymptotics, functional central limit theorems and Poisson approximation theorems for certain functionals that are additive under disjoint unions of graphs. For instance, this includes
        simplex counts and Betti numbers of the Rips complex, as well as general subgraph counts of the random geometric graph. We also present multi-additive extensions that cover the case of persistent Betti numbers of the Rips complex.
    \end{abstract}
    

\end{frontmatter}

\linenumbers

%\maketitle

\section{Introduction}

%MOT
In the random geometric graph, two data points are linked by an edge if their distance does not exceed a chosen threshold. This network lies at the foundation of many advanced clustering methods such as DBSCAN \cite{dbscan}. Hence, asymptotic results on the structure of large random geometric graphs can provide key insights how such algorithms behave on large datasets. Due to the complexity of modern datasets, there is particularly pressing demand of such results in high dimension. This motivates our investigation of high dimensional random geometric graphs.

%SG
Besides this motivation from the application domain, there is also a vibrant research stream of stochastic geometry in high dimensions. This includes general limit results for random tessellations, Boolean models and random polytopes \cite{tessellationBaccelli, booleanBaccelli, gilles2,rsimp,hm}. 
For the random geometric graph, the investigations have so far concentrated on very particular functionals such as edge or clique counts \cite{cliq2,cliq3,cliq1,gry2,gryg}.  While they are certainly fundamental characteristics of a network, the rise of topological data analysis creates the urgent need to understand more refined quantities such as Betti numbers and persistent Betti numbers. In finite dimensions, \cite{crackle,thomas} recently derived central limit theorems (CLTs) and Poisson approximation theorems in the large-volume limit. However, all findings pertain to the model where the dimension remains fixed.

%RESULTS
To address these shortcomings, we develop a general framework for limit results of functionals on the $l_\ff$-random geometric graph in high dimensions. This includes a set of sufficient conditions for CLTs and Poisson approximation theorems (both scalar and functional) on \emph{additive statistics}, i.e., summary statistics that can be computed separately in each of the connected components. One of the key obstacles to arrive at such results in full generality stems from long-range correlations induced by large connected components spanning over macroscopic regions of the sampling window. However, we will not encounter this difficulty since our work will focus on the \emph{sparse regime} where the expected number of neighbors of a typical nodes vanishes in the limit.

%METH
On the methodological side, our main contributions rely on a simple but essential observation. When constructing the random geometric graph with respect to the $l_\ff$-norm, then two vertices of the graph are connected precisely when, on each of the $d$ coordinate axes, the distance between their projections is below the connectivity threshold of the graph. Hence, adding further links into the geometric graph incurs exponential costs so that in the contribution of components with superfluous edges is negligible. This observation makes it possible to derive precise expectation and variance asymptotics in the high-dimensional setting. 

%LIMIT
Once the expectation and variance asymptotics are established, we invoke Stein's method to derive the scalar CLT and Poisson approximation theorem. This blueprint was already successfully implemented in fixed dimensions in \cite{thomas}. The most challenging part in extending the results is caused by the need to control quantities that are no longer of constant order but grow at an exponential speed in the high-dimensional regime. To proceed to a functional CLT, the main challenge is to establish tightness. Here, we rely on the cumulant method, which was successfully implemented in fixed dimensions in  \cite{cyl}. Again, we note that the exponential growth of certain expressions is the main challenge in the high-dimensional regime. 

%Comment from Yogesh
To prove the CLT, we rely on dependency graphs as in \cite[Theorem 2.4]{penrose}. More recently, \cite{lastPeccatiSchulte} developed the Malliavin-Stein calculus to get bounds for a normal approximation. This could also be applied for our purposes. If one were to aim for quantitative normal approximation results, the Malliavin-Stein method would possibly have the advantage of giving better rates of convergence. However, since we do not consider convergence rates, we preferred the dependency graph method due to slightly less involved computations. 
 
%STRUCTURE
The rest of manuscript is organized as follows. First, in Section \ref{mmod_sec}, we introduce the model and state our main results. Next, Section \ref{ex_sec} provides examples for functionals covered by our framework. Finally, Sections \ref{add_proof_sec}, \ref{pat_sec}, \ref{add_mult_sec} and \ref{tight_sec} contain the proofs of the additive CLT, the Poisson approximation theorem, the multi-additive CLT, and the functional CLT, respectively.

\section{Model and main results}
\label{mmod_sec}

\subsection{Model definition}
\label{mod_sec}
%
%RIPS
%
Let $\Rips(\PP_d;s)$ denote the $l_\ff$-Gilbert graph with connectivity radius $s > 0$ constructed on a homogeneous Poisson point process in the torus $W_d := [0, b_d]^d/\sim$ with {intensity} $\la_d^d$. Loosely speaking, the cube $[0, b_d]^d$ is equipped with periodic boundary conditions. The edges of $\Rips(\PP_d;s)$ are the vertex pairs $\{x,y\}\su \PP_d$ with 
$$
|x-y| :=\| x-y\|_{\infty} 
:= \max_{i \le d}(|x_i - y_i|\wedge|x_i - y_i + b_n| \wedge|x_i - y_i-b_n|) \le s,$$
where $\wedge$ denotes the minimum, and where $x, y$ are considered as points in Euclidean space.
The Gilbert graph gives rise to an increasing family of graphs
$$\bigl(\Rips_d(t)\bigr)_{t \le 1} := \bigl(\Rips(\PP_d;t^{1/d})\bigr)_{t \le 1}.$$
The scaling $t^{1/d}$ will be essential for our results to hold. It can be motivated by the fact that the typical degree of this graph has expectation $ (2 \la_d)^d t $, and, in particular is proportional to $t$.

%
%RESTRICTION
%
Our main results concern the structure of the Gilbert graph in high dimensions, i.e., as $d \tff$. We focus on the \emph{sparse regime} where $\la_d\to 0$ and the graph is observed in a cubical sampling window .
In particular, we may assume $\lambda_d< 1/2$, so that the expected number of Poisson points in an $l_\infty$-unit ball is strictly less than $1$. This ensures that components of $\Rips(\PP_d;t^{1/d})$ are almost surely finite for any $t \le 1$. Henceforth, we set
$\Comp_d(t)
:= \{ G \su \Rips_d(t)\co\text{$G$ is a connected component} \}$  {and $|G|$ denotes the number of vertices of an arbitrary graph $G$}.

%
%BETTI
%
\subsection{Additive functionals}
\label{add_sec}
In our first two main results, Theorems \ref{fclt} and \ref{pois_thm} below,
we study the asymptotic behavior of nonnegative functionals on the $l_\infty$-Gilbert graph in high dimensions. That is, for a nonnegative functional $\aaa$ defined on isomorphism classes of abstract graphs, we investigate $$A_{d, t} := \aaa(\Rips_d(t))$$ as a stochastic process on $[0, 1]$ in the limit $d\tff$. To describe precisely the variance-scaling in our CLT, we introduce additional terminology. 

We assume that $\aaa$ is \emph{additive}, i.e., that $\aaa(G \cup G') = \aaa(G) + \aaa(G') $ for any disjoint graphs $G$ and $G'$. We highlight subgraph counts and Betti numbers of the Rips complex as two prominent examples for additive functionals, see Section \ref{add_ex_sec}. 
Then, writing $\mathbb{G}_k$ for the family of all connected graphs on $\{0, \dots, k\}$, we set 
$$\AA_k := \bigl\{G \in \mathbb{G}_k\co \aaa(G) \ne 0\bigr\}\quad \text{ and }\quad k_0 := \min\{k \ge 0 \co \AA_k \ne \es \}
.$$
A key step in the proof of the CLT will be to show that, for any $t\leq 1$, the expectation and variance of the functional $A_{d,t}$ are of order
\begin{equation} \label{eq:defrd}
    \r_d := |W_d| \la_d^{d(k_0 + 1)}v_{\ms{max}}^d.
\end{equation}
Here, $v_{\ms{max}} := \max_{G \in \AA_{k_0}}v(G)$, and, for $G \in \AA_{k}$ with $k \geq 1$,
$$v(G) 
:= \int_{\R^{k}}\one\{G \su \mc{G}_{1, k}(o, u_1, \dots, u_{k}; 1) \}\d ( u_1, \dots, u_{k}),$$
with $o := (0, \dots, 0) \in \R^d$, and where, for $ \xx = (x_0, \dots, x_k) \in \R^{d(k + 1)}$,
$$\RR(\xx; t):=\RRk(\xx; t) := \{ \{i, j\} \su\{0,\dots,k\}\co |x_i-x_j| \le t^{1/d} \}.$$
Here, to simplify the notation, we identify the graph with its edge set.
Note that the vertices of $\RR(\xx; t)$ are integers whereas the vertex set of $\Rips_d(t)$ is contained in $\R^d$. For $k_0 = 0$, we put $v(G) = 1$ and observe that
$$ v(G)^d = \int_{\R^{d k}}\one\{G \su \mc{G}_{d, k}(o, x_1, \dots, x_{k}; 1) \}\d ( x_1, \dots, x_{k}).$$

Before stating the moment asymptotics precisely, we first provide an intuition behind the quantity $\r_d$. To add a connected component consisting of $k + 1$ nodes in the sampling window, we first have to place one of the points inside the window and then insert $k$ further points at a distance of constant order. In expectation, this yields a contribution of order $|W_d|\la_d^{d(k + 1)}$. Hence, observing components with more than the minimal number of $k_0 + 1$ is exponentially unlikely and can be neglected asymptotically. Next, we observe that having an edge between two vertices in the $l_\infty$-Gilbert graph means that when considering the difference between the two vertices, then the absolute value of each of the $d$ coordinates is smaller than the connection threshold. 
Hence, putting additional edges incurs exponential costs. Therefore, configurations which do not realize the maximal value of $v$, and are exponentially unlikely.
We let 
$$\AA^{\ms m}_{k_0} := \bigl\{ G\in\AA_{k_0}\co v(G) = v_{\ms{max}}\bigr\}$$
be the set of configurations $G\in\AA_{k_0}$ such that $v(G) = v_{\ms{max}}$.

For the functional limit results, we consider $A_{d,t}$ as an element of the space of nonnegative càdlàg functions on $[0,1]$ endowed with the Skorokhod topology. We refer the reader to \cite[Section 12]{billingsley} for a detailed introduction of this space. For $(\alpha_d)_d$, $(\beta_d)_d$ nonnegative sequences we write $\alpha_d \sim \beta_d$ if $\lim_{d\to\ff}\alpha_d /\beta_d = 1$.
%FCLT %
\bet[CLT for additive functionals]
    \label{fclt}
     Let $\aaa$ be an additive nonnegative functional with $\aaa(G) \in e^{O(|G|)}$. Moreover, assume that $\la_d \to 0$ and $|W_d|^{1/d}\to \infty$.
    \been
        \im {\bf Moment asymptotics.} 
        Let $0 \le t \le t' \le 1$, then
        $$ \E[A_{d, t}] \sim \f{\r_dt^{k_0}}{(k_0 + 1)!} \sum_{G\in \AA^{\ms m}_{k_0}}\aaa(G) \quad \text{ and } \quad \Cov[A_{d,t'}, A_{d, t}] \sim \f{\r_dt^{k_0}}{((k_0 + 1)!)^2} \sum_{G \in \AA^{\ms m}_{k_0}} \aaa(G)^2.$$
        Moreover, assume that $\aaa$ is increasing. Then, there exists $c_{\ms{inc}}> 0 $ such that for every $d \ge 2$ and $0 \le t \le t' \le 1$ we have 
        $$\E[A_{d,t'} - A_{d, t}] \le c_{\ms{inc}}\r_d(t' - t)
        \quad \text{ and } \quad 
        \Var\bigl[A_{d,t'} - A_{d, t}\bigr] \le c_{\ms{inc}}\r_d(t' - t).$$
        \im {\bf Multivariate CLT.}
        If $\r_d^{1/d}= |W_d|^{1/d} \la_d^{k_0 + 1}v_{\ms{max}}\to\ff$, then,
        as $d \tff$, in the sense of finite-dimensional distributions, we have 
        $$\Var[A_{d,1}]^{-1/2}({A_{d, t} - \E[A_{d, t}]}) \Rightarrow B_{t^{k_0}} .$$
         where $(B_s)_{s \ge 0}$ is standard Brownian motion.%, and 
        \im {\bf Functional CLT.} Assume that $\r_d^{1/d} \to\ff$ and that $\aaa = \aaa^+ - \aaa^-$ with both $\aaa^+, \aaa^-$ increasing {nonnegative} functionals satisfying the growth condition {$\max(\aaa^+(G),\aaa^-(G)) \in e^{O(|G|)}$}.
        Then, as $d \tff$, as a process in $[0, 1]$ and with respect to the Skorokhod topology,
        $$\bigl(\Var[A_{d,1}]^{-1/2}({A_{d, t} - \E[A_{d, t}]})\bigr)_{t \le 1}\Rightarrow \bigl(B_{t^{k_0}}\bigr)_{t \le 1} .$$
    \enen
\ent

Finally, we prove a Poisson-approximation result in the spirit of \cite[Theorem 5.1]{thomas}, but with a different method. We show that, in the regime where the expectations $(\E A_{d,t})_{t\leq 1}$ converge, the process $(A_{d,t})_{t\leq 1}$ converges to a Poisson process.
\bet[Poisson approximation]
    \label{pois_thm}
    Let $\aaa$ be an additive nonnegative functional. Assume that $\la_d \to 0$ and that $\rho_d \to K>0$. Then, as $d \tff$, as a process in $[0, 1]$ and with respect to the Skorokhod topology,
    $$(A_{d, t})_{t\le 1}\Rightarrow \Bigl(\sum_{G \in \AA^{\ms m}_{k_0}} N^{(G)}_t\aaa(G)\Bigr)_{t \le 1},$$
    where $(N^{(G)}_t)_{t \le 1}$, $G \in \AA^{\ms m}_{k_0}$, are independent Poisson processes with expected value $ K t^{k_0} / (k_0 + 1)!$ at time $t$.
\ent
Note that as a corollary, Theorem \ref{pois_thm} also implies Poisson approximation for the finite-dimensional marginals of $(A_{d, t})_{t \le 1}$. Moreover,  a Poisson approximation result with a diverging Poisson parameter can also be a way to derive a CLT \cite[Theorem 3.10]{thoppe}. It appears that such an approach could establish a CLT in a regime where $\rho_d \to \infty$ and $\limsup_{d \to \infty}\rho_d^{1/d} \le c_0$ for a suitable $c_0 = c_0(\aaa) > 1$, see Remark \ref{rem:pclt}.

Note that most of the arguments can be extended to a setting where $W_d$ does not have periodic boundary conditions but is embedded in $\R^d$. However, our proof of the functional CLT needs that for increasing $\aaa$, the process $(A_{d, t})_{t\le 1}$ is increasing in $t$. It is at this point, that we rely on the periodic boundary conditions.

\subsection{Multi-additive functionals}
\label{pair_sec}

The functional CLT from Theorem \ref{fclt} allows to describe the evolution of $A_{d, t}$ as $t \in [0, 1]$. On the other hand, some additive functions already depend by construction on the \emph{joint} configuration of the Gilbert graph at a sequence of $\tb = (t_1, \dots, t_m)$ for $0 \le t_1 \le \cdots \le t_m \le1$. Hence, for an $m$-variate functional $\aaa$ on graphs, we put
\begin{align*}
    A_{d,\tb} := \aaa(\Rips_d(\tb)) ,\;
    \text{where }
    \Rips_d(\tb) := \bigl(\Rips_d(t_1) , \dots, \Rips_d(t_m) \bigr).
\end{align*}

For instance, while the subgraph count captures the number of subgraphs at a fixed $t \in [0, 1]$, the dynamic subgraph count can reflect also a temporal evolution by specifying the configuration of the subgraph at the time vector $\tb$. Moreover, for $m = 2$, persistent Betti numbers on the Rips complex provide another example of a multi-additive functional. We will return to these examples in detail in Section \ref{pair_ex_sec}.

%
%DOM
%

First, we specify what it means for a multivariate functional to be additive.
We proceed similarly to the univariate case by considering the connected components of its last argument.
More precisely, we say that $\aaa$ is \emph{multi-additive} if for every $\Rb = (G_1, \dots, G_m)$ with $G_1 \su \cdots \su G_m$ and $\Rb' = (G'_1, \dots, G'_m)$ with $G'_1 \su \cdots \su G'_m$ and $G_m \cap G_m' = \es$ we have 
$\aaa(\Rb \cup \Rb') = \aaa(\Rb) + \aaa(\Rb')$, where $\Rb \cup \Rb' := (G_1 \cup G'_1, \dots, G_m \cup G'_m)$.

In order to handle the additional complexities in the multivariate setting, we introduce a second condition. More precisely, we say that a nonnegative functional $\aaa$ is \emph{dominated} if there exists $\cdo > 0$ such that
$$ \aaa(G_1, \dots, G_m) \le \cdo \, \aaa(G_m, \dots, G_m) $$
holds for any increasing sequence $ G_1 \su \cdots \su G_m$ of $m$ graphs.
Analogously to the univariate setting, we let $\AA_k$ denote the family of all connected graphs $G$ on the vertex set $\{0, \dots, k\}$ with $\aaa(G, \dots, G) \ne 0$. Then, we define $k_0$ and $\r_d$ as in the univariate setting. 
For $(\alpha_d)_d$, $(\beta_d)_d$ nonnegative sequences we write $\alpha_d \asymp \beta_d$ if $\alpha_d \in O(\beta_d)$ and $\beta_d \in O(\alpha_d)$.

%FCLT
%
\bet[CLT for multi-additive functionals]
    \label{pair_fclt}
    Assume that $\la_d \to 0$ and that $\r_d\to\ff$ as $d\tff$. Further, we assume that the functional $\aaa$ is dominated, nonnegative and multi-additive such that $\aaa(G, \dots, G) \in e^{O(|G|)}$. Let $\tb = (t_1, \dots , t_m)$ with $0 \le t_1 \le \cdots \le t_m \le 1$. Then, as $d \tff$,
    \been
        \im {\bf Moment asymptotics.} 
        $ \E[A_{d,\tb}] \asymp \r_d$ and $\Var[A_{d,\tb}] \asymp \r_d.$
        \im {\bf CLT.}
        $({\Var[A_{d,\tb}]})^{-1/2}({A_{d,\tb} - \E[A_{d,\tb}]})$ converges  to a standard normal random variable.
    \enen
\ent

\begin{remark}
    We do not know whether the result in part 1 of Theorem \ref{pair_fclt} can be sharpened to give the convergence of $\rho_d^{-1}\E[A_{d,\tb}]$. Such a property would be needed to formulate functional central limit or Poisson approximation theorems for the multi-additive case.
\end{remark}

\section{Examples}
\label{ex_sec}

In Sections \ref{add_ex_sec} and \ref{pair_ex_sec}, we provide specific examples for uni- and multivariate functionals satisfying the conditions of Theorems \ref{fclt} and \ref{pair_fclt}, respectively.

\subsection{Additive functionals}
\label{add_ex_sec}

First, we present subgraph counts and Betti numbers as specific examples of additive functionals covered by Theorem \ref{fclt}.
In the following, the clique complex of a graph $G$ is the simplicial complex with vertex set given by the vertex set of $ G$, and where the $q$-simplices consist of all $(q + 1)$-tuples of vertices that form a $(q+1)$-clique. The most prominent example of a clique complex is the Rips complex, which is associated to the geometric graph. 
%
%SUBGOMPLEX
%
\bee[Subgraph counts]
    Let $G_0$ be a fixed graph and define
    $$\aaa(G) := \#\{G' \su G\co G' \cong G_0\}$$
    as the number of subgraphs of $G$ that are isomorphic to $G_0$. Moreover, $\aaa(G)$ satisfies the growth condition $\aaa(G) \le |G|^{|G_0|}$ since any subgraph isomorphic to $G_0$ is determined by choosing the vertices in $G$ that correspond to the vertices in $G_0$. Hence, $\aaa(\cdot)$ satisfies the conditions of Theorem \ref{fclt}. Note also that the $q$-simplex count in the clique complex is a special case with $G_0$ chosen as the complete graph on $q + 1$ vertices.
    
    We next extend the above argumentation to the number
    $$\aaa_{\ms i}(G) := \#\{G' \su_{\ms i} G\co G' \cong G_0\}$$
    of \emph{induced} subgraphs of $G'$ of $G$ that are isomorphic to $G_0$. Here, an induced subgraph $G'$ of $G$ needs to satisfy the constraint that any edge in $G$ whose vertices are contained in $G'$ is also present in $G'$. In particular, for both, the subgraph and the induced subgraph count, we have $k_0 = |G_0| - 1$. 
\ene
%
%BETTI
%
\bee[Betti numbers]
    Let $q \ge 0$ and define 
    $$\aaa(G) := \dim(Z_q(\widehat G)) - \dim(B_q(\widehat G))$$
    to be the $q$th Betti number of the clique complex $\widehat G$ on the graph $G$. Here the increasing functionals $Z_q(\widehat G)$ and $B_q(\widehat G)$ denote the $q$th cycle and boundary spaces of the clique complex, respectively. We refer the reader to \cite{edHar} for a general introduction to simplicial complexes and simplicial homology.
    Then, $\aaa(\cdot)$ satisfies the conditions of Theorem~\ref{fclt} since $0\leq \aaa(G) \leq \dim(Z_q(\widehat G)) \le |G|^q$. As noted in \cite[Example 3.9]{curto2015clique}, we have $k_{0}+1 = 2(q+1)$. In fact, by \cite[Lemma 4.4]{Khale}, $\widehat G$ is isomorphic to the cross-polytope for any $ G \in \AA_{k_0}$.
\ene

%
%PAIR
%
\subsection{Multi-additive functionals}
\label{pair_ex_sec}
Second, we present linear combinations of additive functionals, dynamic subgraph counts and persistent Betti numbers as specific examples of multi-additive functionals covered by Theorem \ref{pair_fclt}.

\bee[Linear combinations of univariate functionals]
    Define
    $$\aaa(G_1, \dots, G_m) := \a_1 \aaa'_1(G_1) + \cdots + \a_m \aaa'_m(G_m),$$
    where $\a_1, \dots, \a_m \ge 0$ and where the $\aaa'_i$ are nonnegative additive functionals. If $\aaa'_i$ are increasing, then $\aaa$ is a dominated multi-additive functional. Moreover, growth bounds on the $\aaa'_i$ translate immediately into growth bounds on $\aaa$. 
\ene

\bee[Dynamic subgraph count]
    Let $\CCC_0 = (G_{0, 1}, \dots, G_{0, m})$ be a fixed sequence of graphs and define
    $$\aaa(G_1, \dots, G_m) := \#\bigl\{G'_1 \su \cdots \su G'_m\co \text{$G'_i \su G_i$ and $G'_i \cong G_{0, i}$ for all $i \le m$}\bigr\}.$$
    Hence, for $\tb = (t_1, \dots, t_m)$, we may think of $(G_{0, 1}, \dots, G_{0, m})$ as a specific motif to be detected in an evolving network. Then, $A_{d,\tb}$ counts the number of times that this motif is found in the filtration $(\Rips_d(t))_{t \le 1}$. By construction, $\aaa$ is dominated, multi-additive and satisfies the growth condition $\aaa(G, \dots, G) \le |G|^{m\cdot|G_{0, m}|}$. We note that the idea of subgraph counts could also be applied to subcomplex counts.
\ene

%
%PERSISTENT BETTI
%
\bee[Persistent Betti numbers]
    Let $q \ge 0$ and define $\aaa(G, G')$ to be the $q$th persistent Betti number associated with the clique complexes of the graphs $G \su G'$. That is, 
    $$\aaa(G, G') := \dim(Z_q(\widehat G)) -\dim( Z_q(\widehat G) \cap B_q(\widehat G')) .$$
    We again refer the reader to \cite{edHar} for more details.
    Then, $\aaa(\cdot, \cdot)$ is dominated with constant $c= 1$, i.e., $\aaa(G, G') \le \aaa(G', G')$ for $G\su G'$. Moreover, $\aaa(G, G') \le |G'|^q$. 
\ene

\section{Proof of Theorem \ref{fclt}, parts 1 and 2}
\label{add_proof_sec}
%
%DEGOMP
%
In Sections \ref{var_sec} and \ref{mult_sec}, we prove parts 1 and 2 of Theorem \ref{fclt}, where in broad strokes we follow the blueprint from \cite{thomas}. The proof of part 3, i.e., of the functional CLT, requires the introduction of substantial machinery and will therefore be deferred to Section \ref{tight_sec}.

\subsection{Expectation and covariance asymptotics}
\label{var_sec}

For the proof of the expectation and covariance asymptotics, we will show the following lemma which is a consequence of the Mecke formula \cite[Theorem 4.7]{poisBook}. For $\xx = (x_1, \dots, x_k) \in W_d^{k}$, $\xx' =(x_1', \dots, x_{k'}') \in W_d^{k'}$ and $t\le 1$, we set $B_t(\xx) := \{y \in W_d\co \min_{i\le k}|y - x_i| \le t^{1/d} \}$, $ \overline W^k_{d, t}:=t^{-1/d}W_d^{k + 1}$ equipped with periodic boundary conditions and $\dist(\xx, \xx') := \min_{i\le k, j \le k'} |x_i - x_j'|$.
\bel[Moment computations]
    \label{lem:Meckeapplied}
    Let $\aaa$ be an additive nonnegative univariate functional. Let $d\ge1$, and $0 \le t\le 1$ and consider the random variable $A_{d, t} = \aaa\bigl(\Rips_d(t) \bigr)$. Moreover, let $A_{d, t, k}$ be the restriction to components of size $k + 1$. Then,
    \been
	\im
        $ \E[A_{d, t}] =\sum_{k\ge0} \la_d^{d (k +1)}({(k+1)!})^{-1}\sum_{G  \in \AA_k} \aaa(G) g_{1, t}(G) ,$
	where
	\begin{align*}
            g_{1, t}(G)
		  &:= t^{ k + 1}\int_{\overline W^k_{d, t}} e^{-t\la_d^d \, |B_1(\xx)|}\one\{G  = \RR(\xx; 1) \} \d \xx.
	\end{align*}
	\im
        $ \Var[A_{d, t}] = \sum_{k,k'\ge0} ({(k+1)!\, (k'+1)!})^{-1} \sum_{\substack{G\in\AA_k \\G'\in\AA_{k'}}} \aaa(G) \aaa(G') w_t(G, G')$, where with $k'' := k+ k'$,
	    $$w_t(G,G') := \one\{G=G'\} \la_d^{d (k +1)}g_{1, t}(G)+ \la_d^{d(k'' + 2)}(g_{2, t}( G,G') - g_{1, t}(G)g_{1, t}(G')),$$
	    and
	    \begin{align*}
            g_{2, t}(G,G')
		  &:=t^{k'' + 2}\int_{\overline W^k_{d, t}} \! \int_{\overline W^{k'}_{d, t}} e^{-t\la_d^d |B_1((\xx, \xx'))|} \one\{ \dist(\xx , \xx')> 1 \}
		  \\&\qquad\times\one\bigl\{G = \RR(\xx; 1), G' = \RR(\xx'; 1) \bigr\} \d\xx'\d \xx,
		\end{align*}
		where $(\xx, \xx') := (x_1,\dots,x_k,x_1',\dots,x_{k'}')$. 
	\im
        $ \Cov[A_{d, t, k},A_{d, t', k}] = ({(k+1)! })^{-2}\sum_{\substack{G, G'\in\AA_k}} \aaa(G) \aaa(G')w_{t, t'}(G,G')$ for $t \le t'$,
	    where we set
	    \begin{align*} 
            w_{t, t'}(G,G') &:= \la_d^{d (k + 1)} g_{1,t, t'}(G, G')+ \la_d^{d (2k+ 2)}(g_{2, t, t'}( G,G') -g_{1, t}(G)g_{1, t'}(G')),
            \intertext{and}
            g_{1, t, t'}(G, G')
            &:=(t')^{k + 1} \int_{ \overline W^k_{d, t'}} e^{-t'\la_d^d \, |B_1(\xx)|}\one\{G = \RR(\xx; t/t'), G' =\RR(\xx; 1) \} \d \xx,
            \intertext{and}
            g_{2, t, t'}(G,G')
            &:=(t')^{2k + 2 } \int_{ \overline W^k_{d, t'}}\int_{ \overline W^{k}_{d, t'}} e^{-t'\la_d^d |B_{t/t'}(\xx) \cup B_1(\xx')|} \one\{ \dist(\xx , \xx') > 1 \}
            \\&\qquad\times \one\bigl\{G = \RR(\xx; t/t'), G' = \RR(\xx'; 1) \bigr\} \d \xx' \d \xx .
	    \end{align*}
    \enen
\enl
\begin{remark} \label{remark:substitution}
    In some instances, it is more practical to work with the following representations of the functions $g_{i,t}$ and $g_{i,t,t'}$, where the first variable is ``fixed'' at the origin.
    For $g_{1,t}$ this is obtained by performing a substitution $(z_0,\dots,z_k) = (x_0,x_1-x_0,\dots,x_k-x_0)$, and then integrating the variable $z_0$ which produces a coefficient $ |W_d|/t$.
    The other cases are derived similarly.
    \begin{align*}
        g_{1, t}(G)
        &= t^{ k}|W_d| \int_{W_d^{k}} e^{-t\la_d^d \, |B_1({(o,\xx)})|}\one\{G = \RR((o,\xx); 1) \} \d \xx,
        \\ g_{2, t}(G,G')
        &= t^{k'' + 1 }|W_d|\int_{W_d^{k }  }\int_{\overline W^{k'}_{d, t}} e^{-t\la_d^d |B_1({((o,\xx), \xx')})|} \one\{ \dist((o,\xx) , \xx') > 1 \}
        \\&\qquad\times \one\bigl\{G = \RR((o,\xx); 1), G' = \RR(\xx'; 1) \bigr\} \d \xx'\d \xx,
        \\ g_{1,t, t'}(G, G')
        &= (t')^{k } |W_d|\int_{W_d^{k }  } e^{-t'\la_d^d \, |B_1({(o,\xx)})|}\one\{G= \RR((o, \xx); t/t'), G' =\RR((o, \xx); 1) \} \d \xx,
        \\ g_{2, t, t'}(G,G')
        &= (t')^{2k + 1 } |W_d|\int_{W_d^{k }  }\int_{\overline W^{k}_{d, t'}} e^{-t'\la_d^d |B_{t/t'}{((o,\xx)) \cup B_1(\xx')}|} \one\{ \dist((o,\xx) , \xx') > 1 \}
        \\&\qquad\times \one\bigl\{G = \RR((o,\xx); t/t'), G' = \RR(\xx'; 1) \bigr\} \d \xx'\d\xx.
    \end{align*}
\end{remark}

Henceforth, we set $s_d(t) := t^{1/d}$.
\bep[Proof of Lemma \ref{lem:Meckeapplied}, part 1]
    First, by additivity,
    $$A_{d, t} = \sum_{k\ge0} \f1{(k+1)!} \sum_{G\in\AA_k} \aaa(G) \sum_{\xx\in\PP^{k+1}_{\ne}} \one\{ \Rips(\xx;s_d(t)) \in \Comp_d( t)\}   \one\{G=\RR(\xx;t)\}, $$
    where the inner sum is taken over all $k + 1$ tuples of pairwise distinct Poisson points.
    Hence, by the Mecke formula \cite[Theorem 4.7]{poisBook},
    \begin{align*}
	    \E[A_{d, t}]
	    &= \sum_{k\ge0}\f{{\la_d^{d (k+1)}}}{(k+1)!} \sum_{G\in\AA_k} \aaa(G) \int_{W_d^{k + 1}} \P\bigl(\Rips(\xx;s_d(t)) \in \Comp_d(\xx; t) \bigr)   \one\{G=\RR(\xx;t)\} \d \xx,
    \end{align*}
    where $\Comp_d(\xx; t)$ is the family of all connected components of $\Rips(\PP_d \cup \xx;s_d(t))$.
    Thus, by the definition of the Gilbert graph and the void probabilities of the Poisson process $\PP_d$, 
    \begin{align*}
	    \E[A_{d, t}]
	    &= \sum_{k\ge0} \f{{\la_d^{d (k+1)}}}{(k+1)!} \sum_{G\in\AA_k} \aaa(G) \int_{W_d^{k + 1}} e^{-\la_d^d |B_{t}(\xx)|} \one\{G=\RR(\xx;t)\} \d \xx .
    \end{align*}
  Implementing the substitution $\zz = \xx/s_d(t) $
  yields the claimed representation of $\E[A_{d, t}]$.
\enp

\bep[Proof of Lemma \ref{lem:Meckeapplied}, parts 2 and 3]
    For the variance, we first notice that $A^2_{d,t}$ can be represented as
    \begin{align*}
        &\sum_{k,k'\ge 0} \f1{(k+1)! \, (k'+1)!} \sum_{\xx\in\PP^{k+1}_{\ne}} \sum_{\xx'\in\PP^{k'+1}_{\ne}} \one\{ \xx \equiv \xx' \} \one\bigl\{ \Rips(\xx;s_d(t)) \in \Comp_d(\xx;t) \bigr\} \aaa(\RR(\xx;t))^2
        \\& + \one\{ \xx \not\equiv \xx' \} \one\{ \Rips(\xx;s_d(t))\in \Comp_d(\xx;t), \Rips(\xx';s_d(t))\in \Comp_d(\xx';t) \} \aaa(\RR(\xx;t)) \aaa(\RR(\xx';t)) ,
    \end{align*}
    where we write $ \xx \equiv \xx'$ for $ \{ x_1 , \dots ,x_k \} = \{ x_1' , \dots ,x_k'\} $.
    Therefore, proceeding as for the expectation of $A_{d, t}$,
    \begin{align*}
        \E[A^2_t]
        &= \sum_{k,k'\ge0} \f1{(k+1)! \, (k'+1)!} \sum_{G\in\AA_k} \sum_{G'\in\AA_{k'}} \aaa(G) \aaa(G') \bigl(\one\{G=G'\} \la_d^{d (k +1)} g_{1, t}(G)+ \la_d^{d (k'' + 2)}g_{2, t}( G,G')\bigr),
    \end{align*}
    and subtracting the expression found for $(\E[A_{d, t}])^2$ gives the claimed expression for the variance. For the covariance, we may proceed along the same lines. To avoid redundancy, we omit the detailed derivation.
\enp

After having established general first- and second-moment formulas, we can now proceed in the vein of \cite[Proposition 6.1 and Theorem 4.1]{thomas} to complete the proof of part 1 of Theorem \ref{fclt}.
We decompose the proof of the theorem into the four proofs below: first, expectation asymptotics; second, uniform bound on the increment's expectation; third, variance asymptotic; fourth, uniform bound on the increment variance.
\bep[Proof of Theorem \ref{fclt}, part 1, expectation]
    First, by Lemma \ref{lem:Meckeapplied},
    \begin{equation}\label{eq:expected_ad}
    \E[A_{d, t}] =\sum_{k\ge0} \f{\la_d^{d (k+1)} }{(k+1)!}\sum_{G \in \AA_k} \aaa(G) g_{1, t}(G).
    \end{equation}
    For brevity, let
    $ S_k := \la_d^{d (k +1)}({(k+1)!})^{-1} \sum_{G \in \mc A_k} \aaa(G) g_{1, t}(G) $
    denote the $k$th summand in \eqref{eq:expected_ad}. We wish to show that the term $S_{k_0}$ dominates the sum. We begin by estimating $g_{1, t}(G)$ for any $G$ with $|G| = k_0+1$. Note that $\la_d^d \, |B_1({(o,\xx)})| \le \la_d^d \, (k_0+1) |B_1(o)| = \la_d^d \, (k_0+1) 2^d \to 0$ because of the sparsity assumption $\la_d \to 0$.
    Therefore, using the representation of $g_{1,t}$ given in Remark \ref{remark:substitution}, we get
    \begin{align}
        \label{bij_eq}
        g_{1, t}(G)
	    & = e^{-\e_{d,G}} t^{k_0}|W_d| \int_{W_d^{k_0 }  } \one\Big\{G = \RR((o,\xx); 1) \Big\} \d \xx,
        \quad 0\le \e_{d,G} \le (k_0+1)2^d \la_d^d.
    \end{align}

    We will see now that as $d \tff$, the equality sign in the indicator may be replaced by an inclusion up to negligible terms.
    Indeed,
    \begin{align*}
	    \one\big\{G \subsetneq \RR((o,\xx); 1)\big\}
	    & = \one\big\{G \subseteq \RR((o,\xx); 1)\big\} \Big( 1 - \prod_{\{i, j\}\not\in E(G)} \one\{ \norm{x_i-x_j} > 1 \} \Big) ,
    \end{align*}
    where $E(G)$ denotes the edge set of $G$ and we agree that $x_0 = o$. But since
    \begin{align*}
        1-\prod_{\{i, j\}\not\in E(G)} \one\{ \norm{x_i-x_j} > 1 \}
	    & = \one \{ \min_{\{i, j\}\not\in E(G)} \norm{x_i-x_j} \le 1 \}
        \le \sum_{\{i, j\}\not\in E(G)} \one \{\norm{x_i-x_j} \le 1 \} ,
    \end{align*}
    we get that
    \begin{align} \label{subsetneq}
	    \int_{W_d^{k_0}  }\hspace{-.1cm} \one\bigl\{ G \subsetneq \RR\bigl((o,\xx); 1\bigr)\bigr\}\d\xx
	    \le\hspace{-.1cm} \sum_{\{i, j\}\not\in E(G)} \int_{W_d^{k_0 }  } \prod_{(i',j')\in E(G)\cup\{\{i, j\}\}} \hspace{-.1cm}\one\{ \norm{x_{i'}-x_{j'}} \le 1 \}\d\xx
        =\sum_{e\not\in E(G)} v(G\cup e)^d.
    \end{align}
    Now, note that, again with the convention $x_0=o$, we have
    \begin{align*}
        v(G)^d = \int_{W_d^{k_0 }  } \prod_{\{i, j\}\in E(G)} \one\{ \norm{x_i-x_j} \le 1 \} \d \xx,
    \end{align*}
    and set
    $\a(G) := \max \left\{ v(G\cup e)/{v(G)} : e \not\in E(G) \right\} $.
    Since $\a(G)<1$ for any $G$, we thus get that
    \begin{align*}
        \int_{ W_d^{k_0}} \one\bigl\{G = \RR\bigl((o,\xx); 1\bigr) \bigr\} \d \xx
        = v(G)^d (1-\e'_{d,G}),\quad
        0\le \e'_{d,G} \le k_0^2 \a(G)^d.
    \end{align*}
    Hence, we obtain from \eqref{bij_eq} that
    $$g_{1, t}(G)
    = e^{-\e_{d,G}} (1-\e'_{d,G}) t^{k_0} |W_d| v(G)^d .$$
    Recalling that $ \r_d = |W_d| \la_d^{d (k_0 + 1)}v_{\ms{max}}^d$, we thus have
    \begin{equation}\label{eq:S_k_0}
	    S_{k_0} 
        \sim \frac{\la_d^{d (k_0 +1)} |W_d| t^{k_0} }{(k_0+1)!} \sum_{G \in \mc A_{k_0}} \aaa(G) v(G)^d
        \sim \frac{\r_d t^{k_0}}{(k_0+1)!}\sum_{G \in \AA_{k_0}^{\ms{m}}}\aaa(G).
    \end{equation}

    %
    %HIGER ORD
    %
    Next, we bound the sum $\sum_{k > k_0} S_k$. For any $G \in \mc A_k$ we use the assumption that $\aaa(G) \le e^{ck}$ for some $c>0$, and the fact that $G$ contains some spanning tree, to get the bound
    $ S_k \le \frac{\la_d^{d (k +1)} e^{ck}}{(k+1)!} \sum_{T \in \mc T_k} \sum_{\substack{G \in \mc A_k\\ G \supseteq T}} g_{1, t}(G) $,
    where we denote by $\mc T_k$ the set of all trees with vertices $\{0, 1, \dots, k\}$.
    Now, we bound the inner sum, for any $T \in \mc T_k$.
    \begin{align*}
        \sum_{\substack{G \in \mc A_k\\ G \supseteq T}} g_{1, t}(G) &\le t^{k}|W_d| \int_{ W_d^{k}} \sum_{\substack{G \in \mc A_k\\ G \supseteq T}} \one\{G = \RR((o,\xx); 1) \} \d \xx
        = |W_d| \int_{W_d^{k }  }\one\{T \su \RR((o,\xx);  1) \}  \d \xx.
    \end{align*}
    Since the last integral is equal to $2^{dk}$, for any tree $T \in \mc T_k$, we may invoke the Cayley formula $\# \mc T_k = (k+1)^{k-1}$ and Stirling's approximation in order to arrive at
    \begin{align}
	    \label{eq:stirling}
	    S_k &\le \frac{\la_d^{d (k +1)} e^{ck}}{(k+1)!} \, |W_d| 2^{dk} (k+1)^{k-1} \le [e(2\la_d)^d]^k \, e^{ck} \, e |W_d|\la_d^d
	    \le M^{d k} \la_d^{d (k +1)}|W_d|
    \end{align}
    for some constant $M$, independent of $d$. Thus,
    \begin{align}
        \label{eq:S_k_rest}
        \sum_{k > k_0} S_{k}\le & |W_d|\la_d^d\sum_{k > k_0} (M^d \la^d)^k = \, \f{|W_d|\la_d^d (M^d \la^d)^{k_0 + 1}}{1-(M^d t\la_d^d)}
	    \le 2M^{k_0 d+ d}d |W_d| \la_d^{d k_0+2d}.
    \end{align}
    Hence, comparing \eqref{eq:S_k_0} and \eqref{eq:S_k_rest} shows that
    $ \E[A_{d, t}] \sim \r_d t^{k_0}((k_0+1)!)^{-1}\sum_{G \in \AA_{k_0}^{\ms m}}\aaa(G).$
\enp

\bep[Proof of Theorem \ref{fclt}, part 1, {uniform expectation bound for increments}]
    Let $E := [t, t']$.
    Similarly as in Lemma \ref{lem:Meckeapplied}, we decompose $A_{d, E} := A_{d,t'} - A_{d,t}$ as $A_{d, E} = \sum_{k \ge k_0} A_{d, E, k}$, where 
    $$ 
    A_{d, E, k}:= \f1{(k + 1)!}\sum_{\xx\in\PP^{k+1}_{\ne}}\one\{ \Rips(\xx;s_d(t')) \in \Comp_d( t')\}(\aaa(\RR(\xx;t')) - \aaa(\RR(\xx;t))).
    $$
    Then, we bound separately the contributions from $\E[A_{d, E, k_0}]$ and from $\E[A_{d, E, k}]$ with $k \ge k_0 + 1$. First, for $k = k_0$, we write 
    $$
        A_{d, E, k_0}= \sum_{\xx\in\PP^{k_0+1}_{\ne}}\one\{ \Rips(\xx;s_d(t')) \in \Comp_d( t')\} \, (h(\xx,t') - h(\xx,t) ),
    $$
    where $h(\xx,t) := \aaa(\RR(\xx;t)) \one\{\RR(\xx;t) \text{ is connected}\}$. Writing $|E|_{(k)} := (t')^k - t^k$ and 
    $s_k := \int_{ W_d^{k}}h((o,\xx),1) \d \xx$ gives that
    \begin{align*}
        \E[A_{d, E, k_0}^{}] &= |W_d|\la_d^{d (k_0 + 1)} \int_{W_d^{k_0}}e^{-\la_d^d \, |B_{t'}((o, \xx))|} (h(\xx,t') - h(\xx,t) ) \d \xx\\
        &\le ((t')^{ k_0} - t^k_0)|W_d|\la_d^{d (k_0 + 1)} \int_{W_d^{k_0}} h((o,\xx),1) \d \xx\\
        &= |W_d|\la_d^{d (k_0 + 1)}|E|_{(k_0)}s_{k_0}.
    \end{align*}
    Moreover, 
    $s_{k_0} \le \sum_{G \in\AA_{k_0}^{\ms m}} \aaa(G)\vm^d \in O(\vm^d)
    $
    implies that $\E[A_{d, E, k_0}^{}] \in O(\r_d|E|)$. 
    
    Next, consider the case where $k \ge k_0 + 1$. Then, we note that if 
    $\xx \in \PP^{k + 1}$ contributes to $A_{d, E, k}$, then there exists 
    a spanning tree $T\in \mc T_k$ such that $T \su \Rips(\xx; s_d(t'))$ and one of the spanning tree edges in $\Rips(\xx; s_d(t'))$ has  a length in $[s_d(t), s_d(t')]$. Therefore, proceeding as in of Theorem \ref{fclt}, we obtain that
    $ \E[A_{d, E, k}] \le M^{dk} |E|_{(k)}|W_d|\la_d^{d(k + 1)}.$
    Moreover, as in \eqref{eq:stirling} we can use that $\RR((o, \xx); 1)$ contains a spanning tree to obtain that $s_k \le  \, c^{d k}$ for a suitable $c > 0$.
    Therefore, 
    \begin{align*}
        \sum_{k \ge k_0 + 1} \E\big[A_{d, E, k}\big] 
        &\le  \f{\r_d}{\vm^d}\sum_{k \ge k_0 + 1} \f{M^{dk}|E|_{(k)} \la_d^{d(k - k_0)} s_k}{(k + 1)!}
        \le \sum_{k \ge k_0 + 1} \frac{|E|_{(k)} \la_d^{d (k - k_0)} (cM)^{d k}}{\vm^{d}} .
    \end{align*}

    Now, we bound the right-hand side, which we henceforth denote as $s_{E}$. 
    Recalling that $|E|_{(k)} = (t')^k-t^k$, we recognize that $s_E$ is the difference of two geometric series. This leads to
    $$s_E =\big(\la_d(cM) ^{k_0 + 1}/ \vm \big)^d \Bigl(\f{(t')^{k_0 + 1}}{1 - t'(cM\la_d)^d} - \f{t^{k_0 + 1}}{1 - t(cM\la_d)^d}\Bigr).$$
    Since the factor on the right-hand side is of order $O(|E|)$, we have $\sum_{k \ge k_0 + 1}\E[A_{d, E, k}^{}]\in O(\r_d |E|)$.    
\enp

\bep[Proof of Theorem \ref{fclt}, part 1, covariance]
    We consider the decomposition $A_{d, t} = A_{d, t, k_0} + A_{d, t, > k_0}$, where $A_{d, t, k_0}$ gathers the contributions of size $k_0$. Then,
    $$\Cov[A_{d, t}, A_{d, t'}] = \Cov[A_{d, t, k_0}, A_{d, t', k_0}] + \Cov[A_{d, t, > k_0},A_{d, t', k_0} ] + \Cov[A_{d, t}, A_{d, t', > k_0}].$$
    We bound the three expressions separately, starting with the first one.
    By Lemma \ref{lem:Meckeapplied}, part 3, we have $ \Cov [A_{d, t, k_0}, A_{d, t', k_0}] =((k_0+1)!)^{-2}( S_= + S_{\ne})$, where $S_= := \la_d^{d (k_0+1)}\sum_{G, G'\in\AA_{k_0}} \aaa(G)\aaa(G') g_{1, t, t'}(G, G')$ and
    \begin{align*}
	    S_{\ne} : = \la_d^{2 d (k_0+1)} \sum_{G,G'\in\AA_{k_0}} \aaa(G) \aaa(G') \, \bigl( g_{2, t, t'}( G, G') - g_{1, t}(G)g_{1, t'}(G')\bigr).
    \end{align*}
    The estimation of the term $S_=$ is almost identical to the computations in \eqref{eq:S_k_0} above, and yields that
    \begin{equation}\label{eq:bound_S_eq}
	    S_= \sim \r_d t^{k_0}\sum_{G \in \AA_{k_0}^{\ms m}} \aaa(G)^2 .
    \end{equation}
    Let us next see that $S_{\ne}$ is negligible in comparison with $S_=$. We have 
    \begin{align*}
	    g_{2, t, t'}( G, G') - g_{1, t}(G)g_{1, t'}(G')
	    &= (t')^{2k_0 + 2}\int_{\overline W^k_{d, t'}}\int_{\overline W^{k}_{d, t'}} \one\bigl\{G = \RR(\xx; t/t'), G' = \RR(\xx'; 1) \bigr\} \\
	    &\quad\! \times\hspace{-4pt}\Bigl[e^{-t'\la_d^d |B_{t/t'}(\xx) \cup B_1(\xx')|} \one\{ \dist(\xx,\xx')>1 \}- e^{-t'\la_d^d \, (|B_{t/t'}(\xx)| + |B_1(\xx')|)} \Bigr] \hspace{-.8235pt} \!\d \xx'\! \d \xx .
    \end{align*}
    Note that the term in the square brackets lies in $[-1, 1]$ and vanishes when the $l_\ff$-distance between $x_0$ and $x_0'$ is more than $2k_0 + 2$. Hence, noting that $t' \le 1$ and that the volume of the set $\{y \in \R^d\co |y - x_0| \le 2k_0 + 2\}$ can be bounded above by $c_0^d$ for some suitable $c_0 > 0$,
    \begin{align*}
        \bigl| g_{2, t, t'}( G, G') - g_{1,t}(G)g_{1, t'}(G') \bigr|
        &\le (t')^{2k_0 + 2} \int_{\overline W^{k_0}_{d, t'}} \int_{\xx'\co \norm{x_0' - x_0} \le 2k_0+ 2}\hspace{-0.6cm} \one\bigl\{G = \RR(\xx; t/t'), G' = \RR(\xx'; 1) \} \d \xx'\d\xx \\
        & \le c_0^{d} |W_d| \int_{\R^{d k_0}} \one\bigl\{G = \RR((0,\xx); t/t') \bigr\} \d \xx \int_{\R^{d k_0}} \one\bigl\{G' = \RR((0,\xx'); 1) \bigr\} \d \xx'.
    \end{align*}
    Thus, noting that we can just ignore $t/t'$ because it is less than 1,
    \begin{align}
    \label{eq:s0c}
        |S_{\ne}| \le c_0^d \la_d^{2 d (k_0 + 1)} |W_d|\Bigl(\sum_{G\in\AA_{k_0}}v(G)^d\Bigr)^2 .
    \end{align}
    Finally, comparing the latter expression with  \eqref{eq:bound_S_eq} gives that
    \begin{align}
        \label{covad_eq}
        \Cov [A_{d, t, k_0}, A_{d, t', k_0}]
        \sim \frac{\r_d  t^{k_0}}{((k_0+1)!)^{2}}\sum_{G \in \AA_{k_0}^{ \ms m}}\aaa(G)^2.
    \end{align}
Next we show that $\Cov[ A_{d, t, > k_0},A_{d, t', k_0}]$ is negligible compared to $\Cov[ A_{d, t, k_0},A_{d, t', k_0}]$. 
Here, by the Cauchy-Schwarz inequality,
    $\Cov[A_{d, t, > k_0},A_{d, t', k_0} ]^2\le \Var
    [A_{d, t, > k_0}] \Var[A_{d, t', k_0}]$. Note that \eqref{covad_eq} applied with $t=t'$ gives that $\Var[A_{d, t', k_0}]$ is of order $\rho_d$, and thus we need to show that
    \begin{align}
        \label{var_low_eq}
        \Var\bigl[A_{d, t, >k_0}\bigr] \in o\bigl(\r_d\bigr).
    \end{align}
    Proceeding similarly as for the first term, we rewrite $\Var[A_{d, t, > k_0}] = S_= + S_{\ne}$, where now
    \begin{align*}
        S_{\ne} & := \sum_{k,k'\ge k_0 + 1} \f{\la_d^{d (k+k'+1)}}{ (k+1)! \, (k'+1)!}\sum_{G\in\AA_k} \sum_{G'\in\AA_{k'}} \aaa(G) \aaa(G') \bigl( g_{2, t, t}(G, G') - g_{1, t}(G)g_{1, t}(G')\bigr),
    \end{align*}
    and $ S_= := \sum_{k \ge k_0 + 1} \la_d^{d (k+1)}((k+1)!)^{-1}\sum_{G\in\AA_k} \aaa(G)^2 g_{1 , t}(G)$.
    From this point on, we again argue as in \eqref{eq:S_k_rest} except that $k_0$ is replaced by $k_0 + 1$. In particular,
    \begin{align}
        \label{sene_eq}
        S_= \sim \r_d\la_d^d\sum_{G \in \AA_{k_0 + 1}} t^{k_0 + 1}\aaa(G)^2 \quad \text{ and }\quad
	|S_{\ne}| \le (M^d t' )^{2k_0 + 2} \r_d\la_d^{d k_0 + d} (1 + o(1)).
    \end{align}
    Hence, we arrive at the asserted $\Var[A_{d, t, > k_0}] \in o\bigl(\r_d\bigr)$. Finally we conclude the proof by observing that with similar argument $\Cov[ A_{d, t},A_{d, t', k_0}]$ is also negligible.
\enp

\bep[Proof of Theorem \ref{fclt}, part 1, {uniform variance bound for the increment}]
    As in the expectation bound, we rely on the decomposition $A_{d, E} = A_{d, E, k_0} + \sum_{k \ge k_0 + 1} A_{d, E, k}$ and start by bounding the $\Var[A_{d, E, k_0}]$. 
    To that end, we note that proceeding as in Lemma \ref{lem:Meckeapplied}, we obtain the decomposition
    $$ \Var[A_{d, E, k_0}] = \f1{((k_0 + 1)!)^2} \E[A^{(1)}_{d, E, k_0}] + {((k_0+1)!)^2} w_{E, k_0, k_0}^{(1)},$$
    where $A^{(1)}_{d, E, k_0}$ is defined as $A_{d, E, k_0}$ except for replacing $\aaa$ by $\aaa^2$ and where
    $$w^{(1)}_{E, k_0, k_0} := \la_d^{d (2k_0 + 2)}g_{2, E}^{(1)}( k_0, k_0) - \E[A_{d, E, k_0}]^2.$$
    Here, setting $h(\xx,t,t') := (\aaa(\RR( \xx;t')) - \aaa(\RR( \xx;t))) \one\{\RR(\xx;t') \text{ is connected}\}$  gives that 
    \begin{align*}
        g_{2, E}^{(1)}(k_0,k_0)
	    &:=\int_{\overline W^{k_0 + 1}_{d, 1}}\int_{\overline W^{k_0 + 1}_{d, 1}} e^{-\la_d^d |B_{t'}((\xx, \xx'))|} \one\{ \dist(\xx, \xx')^d > t' \}
        \, h(\xx,t,t') \, h(\xx,t,t') 
        \d\xx'\d \xx.
    \end{align*}
    Now, note that $B_{t'}((\xx, \xx')) = B_{t'}(\xx) \cup B_{t'}(\xx')$ if $\dist(\xx, \xx') > 2$ and that $|B_{t'}(\xx)| \le (k_0 + 1)2^d$. In particular, an application of Fubini's theorem gives that
    $w_{E, k_0, k_0}^{(1)} \le \la_d^{d (2k_0 + 2)}|W_d|(k_0 + 1)2^d|E|_{(k_0)} s_{k_0}^2$.
    Therefore, recalling that $s_{k_0 } \in O(\vm^d)$ shows that $ w_{E, k,_0 k_0}^{(1)} \in o(|E|\r_d)$. 
    Hence, invoking the bounds derived when computing $\E[A_{d, E, k_0}]$ shows that $\Var[A_{d, E, k_0}] \in O(\r_d |E|)$.
    
    Second, we prove that $\Var[\sum_{k \ge k_0 + 1}A_{d, E, k}] \in O(\r_d |E|)$.  Moreover, we write
    $$ \Var[\sum_{k \ge k_0 + 1}A_{d, E, k}^{}] = \sum_{k \ge k_0 + 1}\f1{((k + 1)!)^2} \E[A^{(1)}_{d, E, k}] + \sum_{k,k'\ge k_0 + 1} \frac1{(k+1)! \, (k'+1)!} w_{E, k, k'}^{(2)},$$
    where 
    $$w^{(2)}_{E, k,k'} := \la_d^{d (k + k' + 2)}g_{2, E}^{(2)}( k, k') - \E[A^{(2)}_{d, E, k}]\E[A^{(2)}_{d, E, k'}].$$
    Here,
    \begin{align*}
        g_{2, E}^{(2)}(k, k')
	    &:=\int_{\overline W^{k + 1}_{d, 1}}\int_{\overline W^{k' + 1}_{d, 1}}e^{-\la_d^d |B_{t'}((\xx, \xx'))|} \one\{ \dist(\xx , \xx')^d > t'\} h(\xx,t,t') h(\xx',t,t') \d\xx' \d \xx.
    \end{align*}
    We note that if $\dist(\xx,\xx') > 2$, then  $B_{t'}((\xx, \xx')) = B_{t'}(\xx) \cup B_{t'}(\xx')$. In particular, an application of Fubini's theorem gives that $w_{E, k, k'}^{(2)}$ is at most
    $$\la_d^{d(k + k' + 2)} |W_d|2^d\Big(\int_{\overline W^{k }_{d, 1}}h((o, \xx),t,t') \d\xx\Big) \Big(\int_{\overline W^{k' }_{d, 1}}h((o, \xx'),t,t') \d\xx' \Big) \le \la_d^{d(k + k' + 2)} |W_d|2^ds_{k }s_{k' }.$$
    Again, recalling that $s_{k } \le (k + 1)! \, (cM)^{d k}$ shows that
    $$ \sum_{k,k'\ge k_0 + 1} \frac1{(k+1)! \, (k'+1)!} w_{E, k, k'}^{(2)}	\le 2^{2d}|W_d||E|\sum_{k \ge k_0 + 1}\la_d^{d (k + 1)}(cM)^{d k}\sum_{k' \ge k_0 + 1}\la_d^{d (k' + 1)}(cM)^{d k'}.$$
    Therefore, $\Var[A_{d, E}^{(2)}] \in O(|E|\r_d)$, as asserted.
\enp

%
%MULT CLT
%
\subsection{Multivariate CLT}
\label{mult_sec}

To prove the multivariate CLT, we proceed in two steps. First, we show that it suffices to consider the functional restricted to components in $\AA_{k_0}$.
Then, the key step in the proof of part 2 of Theorem \ref{fclt} is to establish the following CLT for the restricted functional.
We define $	r(t) := \f{t^{k_0}}{((k_0 + 1)!)^2}\sum_{G\in \AA^{\ms m}_{k_0}} \aaa(G)^2$.

%
%REST CLT
%
\bel[Multivariate CLT for the restricted functional]
    \label{clt_lem}
    Then, as $d \tff$,
    $$\r_d^{-1/2}({A_{d, t, k_0} - \E\bigl[A_{d,t, k_0}\bigr]}) \Rightarrow B_{r(t)},$$
    in the sense of convergence of finite-dimensional marginals. Note that the scaling above implies the scaling appearing in Theorem \ref{fclt}.
\enl

%
%MULT CLT PRF
%
\bep[Proof of Theorem \ref{fclt}, multivariate CLT]
    The covariance asymptotics in \eqref{var_low_eq} imply that
    $$\Var\Bigl[\sum_{k \ge k_0 + 1}A_{d, t, k}\Bigr] \in o(\r_d).$$
    Hence, combining Lemma \ref{clt_lem}, and Chebyshev's inequality concludes the proof.
\enp

Thus, it remains to establish the restricted multivariate CLT in Lemma \ref{clt_lem}. The key idea is to proceed as in \cite{thomas} and rely on Stein's method in the form of \cite[Theorem 2.4]{penrose}.

%
%\CCCLT PRF
%
\bep[Proof of Lemma \ref{clt_lem}]
    In this proof, we will use the following notation,
    $$ A'_{d, t} := \r_d^{-1/2}\bigl({A_{d,t, k_0} - \E\bigl(A_{d,t, k_0}\bigr)}\bigr),\quad
    Y_t := B_{r(t)}, $$
    as well as 
    \begin{align*}
        A_{d, \bs t, k_0} &:= (A_{d, t_1, k_0},\dots,A_{d, t_m, k_0}),\quad &A'_{d,\bs t} &:= (A'_{d,t_1},\dots,A'_{d,t_m}),\quad &Y_{\bs t} &:= (Y_{t_1},\dots,Y_{t_m}),\\
        \cb A_{d, \bs t, k_0} &:= \sum_{i\le m} c_i A_{d, t_i, k_0},\quad &\cb A'_{d,\bs t} &:= \sum_{i \le m} c_i A'_{d,t_i},\quad &\cb Y_{\bs t} &:= \sum_{i\le m} c_i Y_{t_i},
    \end{align*}
    for any $m \ge 1$, $t,t_1,\dots,t_m \le 1$, and $c_1,\dots,c_m \in \R$.
    
    Let $\bs t = (t_1, \dots, t_m)$ with $0 < t_1 < \cdots < t_m \le 1$.
    We have to show that, $ A'_{d, \bs t} \Rightarrow Y_{\bs t} $, as $d\tff$.
    By the Cramér-Wold theorem, it is enough to show that, $\cb A'_{d, \bs t} \Rightarrow \cb Y_{\bs t}$, for any $ c_1,\dots,c_m \in \R$.
    Observe that for any $0< t \le t'\le 1$, $\E[A'_{d, t}] = 0=\E[Y_t]$ and, by the moment asymptotics of Theorem~\ref{fclt},
    \begin{align} \label{eq:covlimit}
        \Cov[A'_{d, t},A'_{d, t'}]
        = \rho_d^{-1} \Cov[ A_{d,t, k_0}, A_{d,t', k_0} ]
        \to \frac{t^{k_0}}{((k_0+1)!)^2} \sum_{G\in \AA^{\ms m}_{k_0}} \aaa(G)^2
        = \Cov[Y_t,Y_{t'}].
    \end{align}
    Note that, by bilinearity of the covariance, Equation \eqref{eq:covlimit} implies that
    \begin{align*}
        \Var (\cb A'_{d, \bs t})
        &= \sum_{i,j} c_i c_j \Cov[ A'_{d, t_i} , A'_{d, t_j} ]
        \to \sum_{i,j} c_i c_j \Cov[ Y_{t_i} , Y_{t_j} ]
        = \Var[\cb Y_{\bs t}].
    \end{align*}
    In particular, if $\Var[\cb Y_{\bs t}] =0$, then $\Var[\cb A'_{d, \bs t}]\to 0$ and $\cb A'_{d, \bs t} \Rightarrow \E \cb A'_{d, \bs t} = 0 = \cb Y_{\bs t}$.
    Therefore, we can assume without loss of generality, that $\Var[\cb A'_{d, \bs t}] \to \Var[\cb Y_{\bs t}] \neq 0$.
    It remains to show that $\cb A'_{d, \bs t} \Rightarrow \cb Y_{\bs t}$, which is equivalent to showing that $(\cb A_{d, \bs t, k_0} - \E \cb A_{d, \bs t, k_0}) /\sqrt{\Var[\cb A_{d, \bs t, k_0}]}$ converges in distribution to a standard normal random variable.
    
    For this we will write $\cb A_{d, \bs t, k_0}$ as a sum of local contributions and apply Stein's method as presented in \cite[Theorem 2.4]{penrose}.
    First, by additivity, 
    $$ \cb A_{d, \bs t, k_0}
    = \sum_{i \le m} c_i A_{d,t_i, k_0}
    = \sum_{\substack{i \le m \\G \in \Comp_d( t_i) \\ |G| = k_0+1} } c_i \aaa\bigl( G\bigr). $$
    For a graph a component $G$ consisting of $k_0 + 1$ vertices, we let $z(G)$ be the lexicographic minimum. Note that this is well-defined since any point set of diameter at most $k_0$ can be considered as a set in the Euclidean space that does not wrap around the torus boundary.
    For $d$ large enough, we can partition the cubical sampling window $W_d$ into subcubes of a side length $\ell = \ell_d$ satisfying $2k_0 \le \ell \le 2k_0 + 1$. Then, we define the set $V_d := \ell \Z^d \cap W_d$, and for $v \in V_d$, we let
    \[\xi_{v, d}
    := \sum_{\substack{i \le m \\ G \in \Comp_d(t_i) \\ |G| = k_0+1} } c_i\one\{z(G) \in(v + [0, \ell]^d) \}\aaa\bigl( G\bigr) \]
    denote the contribution to $\cb A_{d, \bs t, k_0}$ coming from components centered in $v + [0, \ell]^d $.
    Note that if $G$ is as in the last sum (i.e. a connected component with center in $v+[0,\ell]^d$ and cardinality $k_0+1$), then its diameter is at most $k_0 s_d(t_m)$ and thus is contained in the cube $v+[-k_0,\ell+k_0]^d$.
    In particular, it is unaffected by modifications of the point process outside $\cup_{\mathbf{\varepsilon} \in \{-\ell,0,\ell\}^d} v + \varepsilon + [0,\ell]^d $.
    Therefore,
    $ \xi_{v,d} \text{ and } \{ \xi_{w,d} : w-v \not\in \{-\ell,0,\ell\}^d \}$ are independent, for any $v\in V_d$.
    
    We write $\overline\xi_{v, d} := \xi_{v, d} - \E[\xi_{v, d}]$ for the recentered expression and
    $\xi_{v, d}' := {\overline\xi_{v, d}}/{\sqrt{\Var\bigl[ \cb A_{d, \bs t, k_0}\bigr]}}$ for the normalized quantity. Note that we may invoke the variance asymptotics from part 1 of Theorem \ref{fclt} on the functional $\aaa_{k_0}(\cdot)$ that coincides with the original functional $\aaa(\cdot)$ on components of size $k_0 + 1$ and is set to 0 for other components. Therefore, the variance of $ \cb A_{d, \bs t, k_0}$ is of order $\r_d$.
    We are now in the setting to apply Stein's method as presented in \cite[Theorem 2.4]{penrose},
    \begin{align*}
        \Bigl| \P \Bigl( \frac{ \cb A_{d, \bs t, k_0} - \E\bigl[ \cb A_{d, \bs t, k_0}\bigr]}{ \sqrt{\Var \bigl( \cb A_{d, \bs t, k_0}\bigr) }} \le x\Bigr) - \P(Z \le x)\Bigr| &= \Bigl|\P\Bigl( \sum_{v} \xi_{v, d}'\le x\Bigr) - \P(Z \le x)\Bigr|\\
        &\le 2 (2\pi)^{-1/4} \sqrt{D^2 \sum_{v\in V_d} \E[|\xi_{v, d}'|^3] } + 6 \sqrt{D^3 \sum_{v\in V_d} \E[|\xi_{v, d}'|^4]} ,
    \end{align*}
    where $Z$ is a standard normal random variable and $D = |\{-\ell,0,\ell\}^d| = 3^d$.
    Note that $(\xi_{v,d}')_{v\in V_d}$ are identically distributed, and recall that $V_d$ consists of at most $|W_d|$ elements. Thus, there exists a large enough constant $c>0$, such that
    \begin{align*}
        \Bigl|\P\Bigl( \frac{ \cb A_{d, \bs t, k_0} - \E\bigl[ \cb A_{d, \bs t, k_0}\bigr]\bigr)}{ \sqrt{\Var \bigl(A_{d, \bs t, k_0}\bigr) }} \le x\Bigr) - \P(Z \le x)\Bigr|
        &\le \sqrt {\b_{3, d}} + \sqrt{\b_{4, d}},
        & \b_{j, d}
        &:= c^d|W_d| \, \E[|\xi_{o, d}'|^j] .
    \end{align*}
    Now, by definition of $\xi_{o, d}'$, we have that $\E[|\xi_{o, d}'|^j] = \E[|\overline\xi_{o, d}|^j] / \Var\bigl[ \cb A_{d, \bs t, k_0}\bigr]^{j/2} $, and that $\Var[ \cb A_{d, \bs t, k_0}] = \rho_d^{d} \Var[ \cb A'_{d, \bs t, k_0} ] \sim \r_d \Var[ \cb Y_{\bs t, k_0} ] \asymp \r_d$.
    Thus,
    $ \b_{j, d}
    \le \r_d^{-j/2}{|W_d|}c^d \E[|\overline\xi_{o, d}|^j] $, 
    up to a change of the value~$c$.
    To bound $|\overline\xi_{o, d}|$, we let $P:= \PP_d ([-\ell,2\ell]^d)$ be a Poisson random variable with parameter $(3 \ell \la_d)^d$ and note that 
    \begin{align*}
        |\overline\xi_{o, d}|
        &\le \sum_{i \le m} |c_i| \times \big|\{ G \in \Comp( \PP_d;s_d(t_i)) : z(G) \in [0, \ell]^d \,,\, |G| = k_0+1 \}\big| \times \max_{G\in \AA_{k_0}} \aaa(G)
        \\&\le m \max_{ i \le m} |c_i| \times \one\{ \PP_d ([-\ell,2\ell]^d) \ge k_0+1 \} \times \PP_d([-0,\ell]^d) \times \max_{G\in \AA_{k_0}} \aaa(G)
        \\&\le c \, P\, \one\{ P\ge k_0+1 \} ,
    \end{align*}
    where $c$ is some large enough constant.
    Therefore,
    \begin{align*}
        \E[|\overline\xi_{o, d}|^j]
        &\le c \, \E [P^j\one\{ P\ge k_0+1 \} ]
        = c \sum_{k\ge k_0+1} e^{-(3\ell\la_d)^d} \frac{(3 \ell \la_d)^{ dk}}{k!} k^j
        \le (c \la_d)^{d (k_0+1)} ,
    \end{align*}
    where the last step holds after increasing the value of $c$, and is justified since the sum is dominated by its first term. Therefore,
    $ \b_{j, d}
    \le |W_d| \, \r_d^{-j/2} (c \la_d^{k_0+1} )^{{d}}
    = [\rho_d^{(1-j/2)/d}c \,\vm^{-1}]^d
    \to 0 $,
    and hence,
    \begin{align*}
        \Bigl|\P\Bigl( \frac{\cb A_{d, \bs t, k_0} - \E\bigl[\cb A_{d, \bs t, k_0}\bigr]}{ \sqrt{\Var \bigl(\cb A_{d, \bs t, k_0}\bigr) }} \le x\Bigr) - \P(Z \le x)\Bigr| \to 0.
    \end{align*}
\enp

\section{Proof of Theorem \ref{pois_thm}}
\label{pat_sec}
Henceforth, we always assume that $\aaa$ is an additive nonnegative functional and that $\r_d \to K>0$.

This section is organized as follows.
First, we present three auxiliary results, Lemmas \ref{lem:reduc0}--\ref{lem:iidPoisson}, and elucidate how they enter the proof of Theorem \ref{pois_thm}.
Second, based on these auxiliary results, we present a brief proof of Theorem \ref{pois_thm}. Finally, we prove the Lemmas \ref{lem:reduc0}--\ref{lem:iidPoisson}.

As a first step we will see that with high probability, all components contributing to $A_{d,t}$ have $k_0+1$ vertices and satisfy additional properties. To make this precise, we consider the approximate process
$$ A_{d, t, k_0}' 
\coloneqq \sum_{G\in\AA^{\ms m}_{k_0}} \aaa(G) \frac{1}{(k_0+1)!} \sum_{\xx \in \PP_{\ne}^{k_0 + 1} \cap \mathbf{C}'_{d,k_0}} \one\{ \RR(\xx; t) = \RR(\xx; 1) = G \} ,$$
where
$$ \mathbf{C}'_{d,k_0}
\coloneqq \{ \xx \in W_d^{k_0+1} :
\diam(\xx) \le k_0 \},$$
with $\diam(\xx) = \max_{0\leq i < j \leq k} |x_j-x_i|$ for any $\xx = (x_0,\ldots,x_k)\in \R^{d(k+1)}$ and $k\in \N_0$.
\begin{lemma}[Reduction 1: special configurations]
    \label{lem:reduc0}
    With high probability, the processes $(A_{d, t})_{t\le 1}$ and \allowbreak $(A_{d, t, k_0}')_{t\le 1}$ are identical.
\end{lemma}

The index set $\PP_{\ne}^{k_0 + 1} \cap \mathbf{C}'_{d,k_0}$ in the definition of $A_{d, t, k_0}'$ consists of $(k_0+1)$-tuples of distinct Poisson points satisfying some extra condition. 
The next reduction of the problem is an approximation of this sum by a sum over a Poisson process (of $(k_0+1)$-tuples of points). That is, we approximate $A_{d, t, k_0}'$ by a Poisson functional. We will do that by making use of the machinery in \cite{schultePoissonApprox}.
We set
$$ \xi_d \coloneqq \frac{1}{(k_0+1)!} \sum_{\xx \in \PP_{\ne}^{k_0 + 1} \cap \mathbf{C}'_{d,k_0}} \delta_{f(\xx)},$$
where $f\colon W_d^{k_0+1} \to W_d^{k_0+1}$ is an arbitrary measurable map such that the following two properties hold for any $\xx\in W_d^{k_0+1}$ and any permutation $\sigma$ of $\{0,\ldots,k_0\}$: 
1) $\{x_0,\ldots,x_d\} = \{ f_0(\xx) , \ldots , f_d(\xx) \}$, and 
2) $f(\xx) = f(x_{\sigma(0)},\ldots, x_{\sigma(k_0)} )$.
In simple terms, the map $f$ selects for any set of points $\{x_0,\ldots,x_{k_0}\}$ a unique ordering $(x_0,\ldots,x_{k_0})$.
Now, we can represent $A'_{d,t,k_0}$ as follows
$$ A'_{d,t,k_0}
= \sum_{G\in\AA^{\ms m}_{k_0}} \aaa(G) \sum_{\xx \in \xi_d} \one\{ \RR(\xx; t) = \RR(\xx; 1) = G \} .$$
%%%
With \cite[Theorem 3.1]{schultePoissonApprox} we will show that $\xi_d$ can be approximated by a Poisson point process in the space
$$ \widehat{\mathbf{C}}_d 
\coloneqq f(\mathbf{C}_d ) 
= \{ \xx \in W_d^{k_0+1} :  \, f(\xx)=\xx , \, \diam(\xx) \le k_0 \} $$
with intensity measure given by the restriction to $\widehat{\mathbf{C}}_d \subseteq  W_d^{k_0+1} $ of $\lambda_d^{d(k_0+1)}$ times the Lebesgue measure.
%%%
\begin{lemma}[Reduction 2: Poisson process approximation] 
    \label{lem:reduc2}
    There exists a coupling of $\xi_d$  and a homogeneous Poisson point processes $\zeta_{d}$ on $\widehat{\mathbf{C}}_d$ of intensity $\lambda_d^{d(k_0+1)}$ such that $\lim_{d\to\infty} \P( \xi_d = \zeta_d ) =1 $.
\end{lemma}

%%%
Thus, $A'_{d,t,k_0}$ is approximated by the following Poisson functional: 
$$A''_{d,t,k_0}
\coloneqq \sum_{G\in\AA^{\ms m}_{k_0}} \aaa(G) \sum_{\xx \in \zeta_{d}} \one\{ \RR(\xx; t) = \RR(\xx; 1) = G \} ,$$
where $\zeta_d$ is a homogeneous Poisson point process on $\widehat{\mathbf{C}}_d$ of intensity $\lambda_d^{d(k_0+1)}$ such that $\lim_{d\to\infty} \P( \xi_d = \zeta_d ) =1 $.
Finally, we set, for $G\in\AA^{\ms m}_{k_0}$ and $t\in[0,1]$,
$$ N_{d,t}^{(G)} 
\coloneqq \sum_{\xx\in  \zeta_d} \one\{\RR(\xx; t) = \RR(\xx; 1) = G\} . $$
%%%
\begin{lemma}[Independent Poisson processes]
    \label{lem:iidPoisson}
    For any $d\in\N$, the processes $(N^{(G)}_{d,t})_{t \le 1}$, $G \in \AA^{\ms m}_{k_0}$, are independent Poisson processes.
    Moreover, as $d\to \infty$, with respect to the Skorokhod topology,
    $$\bigl(N^{(G)}_{d,t}\bigr)_{t \le 1}\Rightarrow \bigl(N^{(G)}_{t}\bigr)_{t\le 1},$$ 
    where $(N_t^{(G)})_{t\le 1}$ are Poisson processes with expected value $ K t^{k_0} / (k_0 + 1)!$ at time $t$. 
\end{lemma}
%%%
With the help of these lemmas we can prove the Poisson approximation result.
%%%
\begin{proof}[Proof of Theorem \ref{pois_thm}]
    We summarize the discussion above.
    Lemmas \ref{lem:reduc0} and \ref{lem:reduc2} show that the processes $(A_{d,t})_{t \le 1}$ and $ (A''_{d,t,k_0})_{t\le 1}$ coincide with high probability.
    Therefore, Lemma \ref{lem:iidPoisson} yields the proof.
\end{proof}
%%%
\bigskip
The rest of the section is devoted to the proofs of the lemmas.
%%%
\begin{proof}[Proof of Lemma \ref{lem:reduc0}]
    To prove the claim, we consider the events
    \begin{align*}
        E_{d,1}
        &= \left\{ |G| = k_0 + 1 \text{ for all $t \in [0, 1]$ and $ G \in \Comp_d(t)$ with $\aaa(G) > 0$} \right\},
        \intertext{ and }
        E_{d,2} &= \Bigl\{ \RR(\xx;t) = \RR(\xx;1) \in \AA^{\ms m}_{k_0} \text{ and } \Rips(\xx;t) = \Rips(\xx;1) \in\Comp_d(1) \text{ for all $t \in [0, 1]$ and $ \xx \in \PP_{\neq}^{k_0+1}$ with} \\
         &\qquad \text{ and $\RR(\xx;t) \in \AA_{k_0}$} \Bigr\} .
    \end{align*}
    Note that, if $\xx \in \PP_{\neq}^{k_0+1}$ is such that $\GG(\xx;1) \in \Comp_d(1)$, then $\diam(\xx) \le k_0$. 
    Therefore, under the events $E_{d,1}$ and $E_{d,2}$, the processes $(A_{d, t})_{t\le 1}$ and $(A_{d, t, k_0}')_{t\le 1}$ are identical. Hence, it remains to show that $\P(E_{d,1}) \to 1$ and $\P(E_{d,2}) \to 1$ as $d\to\infty$
    
    Assume that the event $E_{d,1}$ does not hold.
    Then there exists $t\in[0,1]$ and $G\in\Comp_d(t)$ with $\aaa(G) > 0$ (which implies $|G|\neq k_0$ by definition of $k_0$) and $|G|\neq k_0+1$. 
    In particular such a component $G$ has cardinality at least $k_0+2$, and we can extract from it a tuple $(x_0,\ldots,x_{k_0+1})$ connected at time $t$. Therefore, 
    \begin{align*}
        E_{d,1}^c
        &\subseteq F_{d,1} \coloneqq \Bigl\{  \diam(\xx) \le k_0 + 1 \text{ for some $\xx \in \PP_{\neq}^{k_0+2} $} \Bigr\} .
    \end{align*}
    Recalling that $|W_d|^{1/d}\tff$ and that $\rho_d=\lambda_d^{d(k_0+1)}|W_d|\vm^d$, we can apply the Mecke formula as in the proof of Lemma \ref{lem:Meckeapplied} to deduce that
    \begin{align*}
        \P( E_{d,1}^c ) 
        \le \P(F_{d,1})
        &\le \E \Bigl[\sum_{\xx\in\PP_{\ne}^{k_0+2}} \one\bigl\{ \diam(\xx) \le k_0 + 1\bigr\}\Bigr]\\
        &= \lambda_d^d |W_d| (2(k_0 + 1) \lambda_d)^{d(k_0+1)} 
        = (2(k_0 + 1))^{d(k_0 + 1)}\vm^{-d}\la_d^d\r_d
        \to 0,
    \end{align*}
    where the limit follows from the assumptions that $\rho_d \to K$ and $\lambda_d\to 0$.

    % \medskip
    It remains to prove that $\P(E_{d,2})\to 1$.
    Assume that the event $E_{d,2}$ does not hold. 
    Then, one of the following events holds:
    \begin{align*}
        F_{d,2}^{(1)} 
        &= \Bigl\{ \RR(\xx;t) \in \AA_{k_0} \setminus \AA^{\ms m}_{k_0} \text{ for some $t \in [0,1]$ and $ \xx \in \PP_{\neq}^{k_0+1}$  }\Bigr\},
        \\F_{d,2}^{(2)} 
        &= \Bigl\{ \RR(\xx;t) \in \AA^{\ms m}_{k_0} \text{ and } \RR(\xx;t) \subsetneq \RR(\xx;1) \text{ for some $t \in [0,1]$ and $ \xx \in \PP_{\neq}^{k_0+1}$ }\Bigr\},
        \\F_{d,2}^{(3)} 
        &= \Bigl\{ \RR(\xx;t) \in \AA^{\ms m}_{k_0} \text{ and } \Rips(\xx;1) \notin\Comp_d(1) \text{ for some $t \in [0,1]$ and $ \xx \in \PP_{\neq}^{k_0+1}$ }\Bigr\}.
    \end{align*}
    
    Under the event $F_{d,2}^{(1)}$ there exists a connected graph $G\in\mathbb{G}_{k_0}$ with $v(G)<\vm$ and a tuple $(x_0,\dots,x_{k_0})\in\PP_{\neq}^{k_0+1}$ and with $G\subseteq\RR(\xx;1)$, or equivalently,
    \begin{align*}
        F_{d,2}^{(1)} \subseteq F'_{d,2} \coloneqq \left\{G\subseteq\RR(\xx;1)\text{ for some $G\in\mathbb{G}_{k_0}$ and $(x_0,\ldots,x_{k_0}) \in \PP_{\neq}^{k_0+1}$\! with   $v(G)<\vm$  }\! \right\}\! .
    \end{align*}
    Again, with the help of Mecke's formula, we get
    \begin{align*}
        \P(F'_{d,2}) &\le \sum_{\substack{G\in\mathbb{G}_{k_0} \\ v(G)<\vm}} \E \Bigl[\sum_{\xx\in\PP_{\ne}^{k_0+1}} \one\bigl\{  G\subseteq\RR(\xx;1) \bigr\}\Bigr] 
        =  \frac{\rho_d}{\vm^d} \sum_{\substack{G\in\mathbb{G}_{k_0} \\ v(G)<\vm}} \int_{\R^{d k_0}} \one\{G\subseteq \RR((o,\xx); 1)\} \d\xx .
    \end{align*}
    Note that the last indicator function equals $\prod_{(i,j)\in E(G)} \one\{ \norm{x_{i}-x_{j}} \le 1 \}$ which integrates to $v(G)^d$, by definition of $v(\cdot)$. Thus,
    \begin{align*}
        \P(F_{d,2}^{(1)}) 
        &\le \P(F'_{d,2})
        \le \sum_{\substack{G\in\mathbb{G}_{k_0} \\ v(G)<\vm}}\rho_d \Bigl(\frac{v(G)}{\vm}\Bigr)^d \to 0,
    \end{align*}
    where the convergence follows since $v(G)/\vm<1$.

    % \medskip
    Under $F_{d,2}^{(2)}$, we can consider $t\in[0,1]$ and $\xx \in \PP_{\neq}^{k_0+1}$ such that $\RR(\xx;t)\in\AA^{\ms m}_{k_0}$ and $\RR(\xx;t) \subsetneq \RR(\xx;1)$.
    In particular, there are two vertices of $\Rips(\xx;t)$ that are not connected at time $t$ but get connected by time~$1$.
    Thus, $\vm = v(\RR(\xx;t)) > v(\RR(\xx;1))$ and $F'_{d,2}$ holds.
    Therefore,
    $ \P(F_{d,2}^{(2)})
    \le \P(F'_{d,2}) 
    \to 0 .$

    % \medskip
    Under $F_{d,2}^{(3)}$ and we can consider $t\in[0,1]$ and $\xx \in \PP_{\neq}^{k_0+1}$ such that $\RR(\xx;t)\in\AA^{\ms m}_{k_0}$ and $\Rips(\xx;t)\notin \Comp_d(1)$.
    That implies that the connected component to which $x_0, \ldots , x_{k_0}$ belong contains additional points, and therefore $F_{d,1}$ is satisfied.
    Hence,
    $\P(F_{d,2}^{(3)}) 
    \le \P(F_{d,1})
    \to 0 $.
    
    Therefore,
    $ \P(E_{d,2}) 
    \geq 1 - \P( F_{d,2}^{(1)}) - \P(F_{d,2}^{(2)}) - \P(F_{d,2}^{(3)}) 
    \to 1 .$ 
\end{proof}

%%%
%%%
\begin{proof}[Proof of Lemma \ref{lem:reduc2}]
    Let $\zeta_d$ be a Poisson point process $\widehat{\mathbf{C}}_d \subseteq W_d^{k_0+1}$ of intensity measure $\lambda_d^{d(k_0+1)}$ times the Lebesgue measure.
    The statement of Lemma \ref{lem:reduc2} is equivalent to saying that $d_{\ms{KR}}(\xi_d,\zeta_d) \to 0$, where $d_{\ms{KR}}$ denotes the Kantorovich-Rubinstein distance.
    We refer to \cite{schultePoissonApprox} for a precise definition of this distance.
    For our purpose it is enough to know that this distance gives an upper bound for the probability $\P(\xi_d \ne \zeta_d)$ for an optimal coupling between $\xi_d$ and a homogeneous Poisson point process $\zeta_d$.

    Observe that, by construction, $\xi_d$ is a point process on $\widehat{\mathbf{C}}_d \subseteq W_d^{k_0+1}$ of intensity measure $\lambda_d^{d(k_0+1)}$ times the Lebesgue measure.
    This is precisely the intensity measure of $\zeta_d$.
    Thus, by Theorem 3.1 in \cite{schultePoissonApprox},
    $$ d_{\ms{KR}}(\xi_d,\zeta_d) 
    \le \frac{2^{k_0+2}}{(k_0+1)!} r(\widehat{\mathbf{C}}_d) , $$
    where
    \begin{align*}
        r(\widehat{\mathbf{C}}_d) 
        &\coloneqq \max_{0\le \ell\le k_0-1} \int_{W_d^{\ell+1}} \bigg(\int_{W_d^{k_0-\ell}} \one\{ \xx \in \widehat{\mathbf{C}}_d \} \lambda_d^{d(k_0-\ell)} \d (x_{\ell+1},\dots,x_{k_0})\bigg)^2 \lambda_d^{d(\ell+1)} \d (x_0,\dots,x_\ell) .
    \end{align*}
    Recall that
    $$ \widehat{\mathbf{C}}_d 
    \coloneqq f(\mathbf{C}_d ) 
    = \{ \xx \in W_d^{k_0+1} :  \, f(\xx)=\xx , \, \diam(\xx) \le k_0 \} .$$
    In particular,
    $$ \widehat{\mathbf{C}}_d 
    \subseteq \{ (x_0,\xx) \in W_d \times W_d^{ k_0} :\, \xx - x_0 \in [-k_0,k_0]^{d k_0} \} ,$$
    and thus, for all $d$ satisfying $2 k_0 \lambda_d< 1$,
    \begin{align*}
        r(\widehat{\mathbf{C}}_d) 
        &\le \max_{0\le \ell\le k_0-1} |W_d| \lambda_d \times (2 k_0 \lambda_d)^{d \ell} \times (2 k_0 \lambda_d)^{2 d (k_0-\ell)}
        = |W_d| \lambda_d^{d} (2k_0 \lambda_d)^{d (k_0+1)}.
    \end{align*}
    Recalling that $\rho_d = |W_d| \lambda_d^{d (k_0+1)} \vm^d \to K$, we have
    \begin{align*}
        r(\widehat{\mathbf{C}}_d) 
        &\le \rho_d \left( \frac{\lambda_d (2k_0)^{k_0+1}}{\vm} \right)^d 
        \in o(1)^d 
        \subseteq o(1),
    \end{align*}
    which implies that $d_{\ms{KR}}(\xi_d,\zeta_d) \to 0 $, thereby concluding the proof.
\end{proof}

%%%
\begin{proof}[Proof of Lemma \ref{lem:iidPoisson}]
    For $G\in\AA^{\ms m}_{k_0}$ we set 
    $$\zeta_G \coloneqq  \sum_{\xx \in \zeta} \one\{ \RR(\xx;1) = G \} \, \delta_{\xx},  $$
    which is the restriction of $\zeta$ to the set
    $$ \widehat{\mathbf{C}}_d 
    \coloneqq f(\mathbf{C}_d ) 
    = \{ \xx \in W_d^{k_0+1} :  \, f(\xx)=\xx , \, \RR(\xx;1) = G \} .$$
    Therefore, similarly as for $\zeta$, the restricted process $\zeta_G$ is a Poisson point process on $\widehat{\mathbf{C}}_{d,G}
    \subseteq \R^{d (k_0+1)}$ and has intensity measure $\lambda_d^{d (k_0+1)}$ times the Lebesgue measure.

    We will now map $\zeta_G$ to a Poisson process in $[0,1]$ by considering, for each $\xx\in\zeta_G$ the first time $t$ for which $\RR(\xx ; t) = G$.
    For this, we define
    $$ \tau_G 
    \coloneqq \sum_{\xx \in \zeta_G} \delta_{\min \{ t\in [0,1] : \RR(\xx ; t) = G\}} . $$
    By the mapping theorem for Poisson processes, this is a Poisson process in $[0,1]$.
    Note that, for any $\xx\in\zeta_G$ and any $t\in[0,1]$, we have $\RR(\xx; t) = G $ if and only if $t\geq \min \{ t'\in [0,1] : \RR(\xx ; t') = G\}$. Thus
    \begin{align}
    \label{eq:2907a}
        N_{d,t}^{(G)} 
        = \sum_{\xx\in  \zeta_G} \one\{\RR(\xx; t) = G \}
        = \tau_G ([0,t]). 
    \end{align}
    In particular $(N_t^{(G)})_{t\le 1}$ is a Poisson process.

    Moreover, we observe that $\{ \zeta_G : G \in \AA^{\ms m}_{k_0} \}$ forms a collection of independent Poisson process, because they are the restrictions of a single Poisson process to pairwise disjoint sets.
    Thus, we immediately get the independence of the processes $(N^{(G)}_{d,t})_{t \le 1}$, $G \in \AA^{\ms m}_{k_0}$.

    It remains only to compute the intensity measure of $(N^{(G)}_t)_{t \le 1}$.
    With the Mecke formula we get
    \begin{align*}
        \E \Bigl[\sum_{\xx\in \zeta_G} \one\{\RR(\xx; t) = G \}\Bigr]
        &= \int_{W_d^{k_0+1}} \one\{ f(\xx)=\xx , \, \RR(\xx;1) = G \} \,\lambda_d^{d(k_0+1)} \d (x_0,\ldots, x_{k_0})
        \\&= \frac{\lambda_d^{d(k_0+1)}}{(k_0+1)!} \int_{W_d^{k_0+1}} \one\{ \, \RR(\xx; t) = G \} \d (x_0,\ldots, x_{k_0}) .
    \end{align*}
    Therefore, with \eqref{eq:2907a}, we find
    \begin{align*}
        \E[N^{(G)}_{d,t}]
        &= \frac{|W_d|\lambda_d^{d(k_0+1)} }{(k_0+1)!} \int_{W_d^{k_0}} \one\{ \RR((o,\xx); t) = G \} \d \xx
        = \frac{|W_d|\lambda_d^{d(k_0+1)} t^{k_0}}{(k_0+1)!} \int_{W_d^{k_0}} \one\{ \RR((o,\yy); 1) = G \} \d \yy ,
    \end{align*}
    where in the last equality we performed the substitutions $x_i = s_d(t) y_i = t^{1/d} y_i$. Now, we point out that in the proof of Theorem \ref{fclt}, see Equation \eqref{subsetneq}, it was verified that
    $$\int_{W_d^{k_0}} \one\{ \RR((o,\yy); 1) = G \} \d\yy
    \sim \int_{W_d^{k_0}} \one\{ \RR((o,\yy); 1) \supseteq G \} \d\yy = v(G)^d.$$
    Therefore, noting that $G$ was chosen from $\AA^{\ms m}_{k_0}$ yields the asserted  convergence
    $\E[N^{(G)}_{d,t}]
    \sim \frac{\rho_d t^{k_0}}{(k_0+1)!}
    \to \frac{K t^{k_0}}{(k_0+1)!}.$
    By \cite[Theorem 4.33]{kallen}, the convergence of the intensity measures implies the weak convergence of the associated processes $(N_{d, t}^{(G)})_{t \le 1}$, thereby concluding the proof.
\end{proof}

\begin{remark}
    \label{rem:pclt}
    Note that the proof of Lemma \ref{lem:reduc0} extends to the regime where $\rho_d \to \infty$ and $\limsup_{d \to \infty}\rho_d^{1/d} < \inf\bigl\{{\vm}/v(G)\co G\in\mathbb{G}_{k_0}, v(G)<\vm\bigr\}$. Moreover, as a small side note, we would like to point out that under some circumstances it is possible to retrieve a CLT from the Poisson approximation. In particular, if $\rho_d\to\infty$ but $\rho_d \lambda_d^d \to 0$ and therefore, the total variation between $\xi_d$ and a Poisson point process tends to $0$, proceeding along the lines of \cite[Theorem 3.10]{thoppe} would give a shortcut to a normal approximation.
\end{remark}

\section{Proof of Theorem \ref{pair_fclt} -- multi-additive functionals}
\label{add_mult_sec}

In this section, we prove Theorem \ref{pair_fclt}. To that end, we first derive the expectation and variance asymptotics in Section \ref{add_var_sec}, and then establish the actual CLT in Section \ref{pair_mult_sec}.

\subsection{Expectation and variance asymptotics }
\label{add_var_sec}

Again, the first step in the proof of the expectation and  variance asymptotics is an expansion based on the Mecke formula. First, we generalize the definition of $\RR$ to multiple radii by setting
$$\RR(\xx; \tb) := \bigl(\{ \{i, j\} \su \{0,\dots,k\}\co \norm{x_i-x_j} \le s_d(t_\ell) \}\bigr)_{\ell \le m},$$
where $\tb = (t_1, \dots, t_m) \in [0, 1]^m$ and $\xx = (x_0, \dots, x_k) \in W_d^{k + 1}$.
We also let $\II_k$ denote the family of all $m$-tuples $\CCC$ of graphs on the vertex set $\{0, \dots, k\}$ with $\aaa(\CCC) \ne0$.

\bel \label{lem:Meckeappliedpair}
    Let $\aaa$ be an $m$-variate nonnegative multi-additive functional on graphs. For $d\ge1$ and $\tb = (t_1, \dots, t_m)$ with $0 \le t_1 \le \cdots \le t_m \le 1$ let
    $$A_{d,\tb} := \aaa\bigl(\Rips_d(\tb)\bigr),$$
    then
    \been
        \im
        $ \E[ A_{d,\tb}] = \sum_{k\ge0} ((k+1)!)^{-1}\la_d^{d (k + 1)}\sum_{\CCC \in \II_k} \aaa(\CCC) g_{1, \tb}(\CCC)$, where
        \begin{align*}
            g_{1, \tb}(\CCC)
            &= t_m^{ k}|W_d| \int_{W_d^{k}} e^{-t_m\la_d^d \, |B_1({(o,\xx)})|}\one\{\CCC  = \RR((o,\xx);\tb/t_m) \} \d \xx.
        \end{align*}
        \im
        $ \Var( A_{d,\tb}) = \sum_{k,k'\ge0}  ({(k+1)! \, (k'+1)!})^{-1} \sum_{\substack{\CCC\in\II_k \\\CCC'\in\II_{k'}}} \aaa(\CCC) \aaa(\CCC') w_{\tb}(\CCC, \CCC')$, where
        $$w_{\tb}(\CCC,\CCC') := \one\{\CCC=\CCC'\}\la_d^{d (k + 1)} g_{1, \tb}(\CCC)+ \la_d^{d (k + k' + 2)}\bigl(g_{2,\tb}( \CCC,\CCC') - g_{1, \tb}(\CCC)g_{1, \tb}(\CCC')\bigr),$$
        and
        \begin{align*}
            g_{2, \tb}(\CCC,\CCC')
            &= t_m^{k + k' + 2} \int_{\overline W_{d, t_m}^k} \int_{\overline W_{d, t_m}^{k'}} e^{-t_m\la_d^d |B_1({(\xx, \xx')})|} \one\{ \dist(\xx , \xx') > 1 \}
            \\&\qquad\times  \one\bigl\{\CCC  = \RR(\xx; \tb/t_m), \CCC' = \RR(\xx'; \tb/t_m) \bigr\} \d \xx' \d \xx .
        \end{align*}
    \enen
\enl

\bep
    Since the steps are analogous to those presented in Lemma \ref{lem:Meckeapplied}, we omit the proof.
\enp

After having established general first- and second-moment formulas, we can now proceed to deriving the asymptotics asserted in part 1 of Theorem \ref{pair_fclt}.

%
%EXP PROF
%
\bep[Proof of Theorem \ref{pair_fclt}, part 1, expectation]
    Since the arguments for multi-additive functionals are very similar to additive functionals, we only discuss the main differences.
    First, by part 1 of Lemma \ref{lem:Meckeappliedpair},
    \begin{equation*}
	    \E[A_{d,\tb}] = \sum_{k \ge k_0} \frac{\la_d^{d (k + 1)}}{(k+1)!} \sum_{\CCC \in \mc A_k^{(m)}} \aaa(\CCC) g_{1, \tb}(\CCC).
    \end{equation*}
    In particular, writing $\AA_k := \{G : |G|=k+1, a(G,\dots,G)>0\}$,
    $$\E[A_{d, \tb}] \ge \f{\la_d^{(k_0 + 1)d}}{(k_0 + 1)!} \sum_{G \in \AA_{k_0}} \aaa(G, \dots, G)g_{1,\tb}(G,\dots,G).$$
    To derive the asymptotics of $g_{1, \tb}(G, \dots, G)$, we note that $\one\bigl\{ G \su \RR\bigl((o,\xx); t_1/t_m\bigr)\bigr\}$ is at most
    $$ \one\bigl\{ G = \RR\bigl((o,\xx); t_1/t_m\bigr) = \cdots = \RR\bigl((o,\xx); t_{m}/t_m\bigr)\bigr\} + \one \{G \subsetneq \RR\bigl((o,\xx); 1\bigr)\}.$$
    Moreover, assuming $d$ is sufficiently large to guarantee that $|W_d|\ge (2 k_0)^d$,
    $$\int_{W_d^{k_0}} \one\bigl\{ G \su \RR\bigl((o,\xx); t_1/t_m\bigr)\bigr\}\d\xx = (t_1/t_m)^{k_0}v(G)^d.$$
    On the other hand, we recall from \eqref{subsetneq} in the proof of Theorem \ref{fclt} that
    $$ \int_{W_d^{k_0}} \one\bigl\{ G \subsetneq \RR\bigl((o,\xx); 1\bigr)\bigr\}\d\xx
    \le \sum_{e\not\in E(G)} v(G\cup e)^d. $$
    In particular, $\liminf_{d\tff} \E[A_{d,\tb}]/\r_d > 0$. For the upper bound, we note that for any $k \ge k_0$,
    $$ \sum_{\CCC \in \II_k} \aaa(\CCC) g_{1, \tb}(\CCC) \le  c_\ms{Dom} \sum_{G \in \mc A_k} g_{1, t_m}(G)\aaa(G, \dots, G),$$
    because for any $G\in\AA_k$, we have
    \begin{align*}
        \sum_{\CCC \in \II_k : G_m=G} \aaa(\CCC) g_{1, \tb}(\CCC)
        &= \sum_{\CCC \in \II_k : G_m=G} \aaa(\CCC) t_m^{ k}|W_d| \int_{W_d^{k}} e^{-t_m\la_d^d \, |B_1({(o,\xx)})|} \one\{\CCC = \RR((o,\xx);\tb/t_m)\} \d \xx
        \\&\leq  c_\ms{Dom} \aaa(G,\ldots,G) t_m^{ k}|W_d| \int_{W_d^{k}} e^{-t_m\la_d^d \, |B_1({(o,\xx)})|} \one\{G = \RR((o,\xx);1)\} \d \xx
        \\&= c_\ms{Dom} \aaa(G,\ldots,G) g_{1, t_m}(G).
    \end{align*}
    After setting  $\aaa'(G) \coloneqq \aaa(G, \dots, G)$, we cite the steps in the upper bound in the univariate case.
\enp

The variance asymptotics is obtained through similar steps. Nevertheless, for the reader's convenience, we present the main steps.

%
%VAR PROOF
%
\bep[Proof of Theorem \ref{pair_fclt}, part 1, variance]
    First, by part 2 of Lemma \ref{lem:Meckeappliedpair},
    \begin{align*}
	    \Var (A_{d, \tb}) &= \sum_{k,k'\ge k_0} \frac{1}{(k+1)! \,(k'+1)!}\sum_{\CCC \in \II_k} \sum_{\CCC' \in \II_{k'}} \aaa(\CCC) \aaa(\CCC') w_{\tb, \tb'}(\CCC, \CCC') .
    \end{align*}
    We rewrite this as $\Var (A_{d, \tb}) = S_= + S_{\ne}$, where
    \begin{align*}
        S_{\ne} & := \sum_{k,k'\ge k_0 } \f{\la_d^{d (k + k' + 2)}}{(k+1)! \, (k'+1)!}\sum_{\CCC\in\II_k} \sum_{\CCC'\in\II_{k'}} \aaa(\CCC) \aaa(\CCC') \bigl( g_{2, \tb, \tb}(\CCC, \CCC') - g_{1, \tb}(\CCC)g_{1, \tb}(\CCC')\bigr),
    \end{align*}
    and $S_= := \sum_{k \ge k_0 } ((k+1)!)^{-1}\la_d^{d (k + 1)}\sum_{G\in\II_k} \aaa(\CCC)^2 g_{1, \tb}(\CCC).$
    We first derive upper bounds for $S_=$ and $S_{\ne}$ separately, starting with $S_=$. To that end, writing $\AA_k := \{G : |G|=k+1, \, \aaa(G,\dots,G)>0\}$,    we proceed as in the expectation scaling to obtain that 
    $$\sum_{\CCC \in \II_k} \aaa(\CCC)^2 g_{1, \tb}(\CCC) \le { c_\ms{Dom}} \sum_{G \in \mc A_k} \aaa(G, \dots, G)^2g_{1, t_m}(G).$$
    Hence, repeating the steps leading to \eqref{sene_eq}, we arrive at $ S_= \in O(\r_d)$. 
    Moreover, arguing as in the derivation of \eqref{eq:s0c},
    \begin{align*}
        &\sum_{\CCC\in\II_k} \sum_{\CCC'\in\II_{k'}} \aaa(\CCC) \aaa(\CCC') \bigl( g_{2, t_m, t_m}(\CCC, \CCC') - g_{1, t_m}(\CCC) g_{1, t_m}(\CCC')\bigr) \\
        &\quad \le \sum_{\CCC\in\II_k} \sum_{\CCC'\in\II_{k'}} \aaa(\CCC) \aaa(\CCC')c_0^d|W_d|\int_{W_d^{k}} \int_{W_d^{k'}} \one\bigl\{\CCC = \RR((o, \xx); \tb), \CCC' = \RR((o, \xx'); \tb) \bigr\}\d \xx' \d \xx\\
        &\quad \le c_{\ms{Dom}}^2c_0^d|W_d|\Bigl(\sum_{G\in\AA_k}\aaa(G, \dots, G)v({ G})^d\Bigr)\Bigl( \sum_{G'\in\AA_{k'}} \aaa(G', \dots, G') v(G')^d\Bigr)
    \end{align*}
    for a suitable $c_0 > 0$. Thus, 
    $$ |S_{\ne}| \le c_{\ms{Dom}}^2c_0^d |W_d| \Bigl(\sum_{k\ge k_0 } \f{\la_d^{d (k + 1)}}{(k+1)!}\sum_{\CCC\in\AA_k} \aaa(G, \dots, G) v(G)^d\Bigr)^2.$$
    Hence, again repeating the steps leading to \eqref{sene_eq} gives that $|S_{\ne}| \in o(\r_d)$, thereby showing that $\Var(A_{d, \tb}) \in O(\r_d)$.
    Finally, for the lower bound, we obtain that
    $$S_= \ge \f{\la_d^{d (k_0 + 1)}}{(k_0 + 1)!}\sum_{G\in\AA_{k_0}} \aaa(G, \dots, G)^2 g_{1, \tb}(G, \dots, G),$$
    which is of the order $\r_d$ by the derivations in the univariate setting.
\enp

%
%CLT
%
\subsection{CLT}
\label{pair_mult_sec}

To prove the CLT for multivariate additive functionals, we adapt the strategy from Section \ref{mult_sec}. That is, we first reduce the task to establishing the CLT for the functional restricted to components in degree $k_0$. Then, we apply Stein's method in order to establish the CLT in the restricted setting. More precisely, we define the truncated functional
$$ A_{d,\tb,k_0} := \sum_{\substack{(G_1, \dots, G_m) \in \Comp( \Rips_d(\tb)) \\ |G_m| = k_0 + 1}} \aaa(G_1, \dots, G_m),$$
where
$$\Comp( \Rips_d(\tb)):= \Bigl\{\bigl(\Rips_d(t_1) \cap G , \dots, \Rips_d(t_{m - 1}) \cap G, G \bigr)\co \text{$G$ is a component of $\Rips_d(t_m)$ } \Bigr\}.$$

%
%TRUNC CLT
%
\bel[CLT for truncated multivariate additive functionals]
    \label{trunc_pair_lem}
    As $d \tff$,
    $$\f{{A_{d,\tb,k_0} - \E\bigl[A_{d,\tb,k_0}\bigr]}}{\sqrt{\Var(A_{d,\tb,k_0})}} \Rightarrow \mc N(0, 1),$$
    where the right-hand side denotes a standard normal random variable.
\enl

\bep[Proof of Theorem \ref{pair_fclt}, part 2]
    As in the proof of the variance asymptotics in part 1 of Theorem \ref{pair_fclt}, 
    $$\Var\Bigl(\sum_{k \ge k_0 + 1}A_{d,\tb,k}\Bigr) \in o(\r_d).$$
    Thus, combining Lemma \ref{trunc_pair_lem} with the Chebyshev inequality concludes the proof.
\enp

To prove Lemma \ref{trunc_pair_lem}, we again proceed in the vein of \cite[Proposition 6.1 and Theorem 4.1]{thomas}.
%
%\CCCLT PRF
%
\bep[Proof of Lemma \ref{trunc_pair_lem}]
    We write $A_{d,\tb,k_0}$ as a  sum of local contribution and apply Stein's method as presented in \cite[Theorem 2.4]{penrose}.
    Then, by definition,
    $$ A_{d,\tb,k_0}
    =	 \sum_{\substack{(G_1, \dots, G_m) \in \Comp( \Rips_d(\tb)) \\ |G_m| = k_0 + 1}} \aaa(G_1, \dots, G_m). $$
    We omit the rest of the proof since it is identical to that of Lemma \ref{clt_lem} after replacing	$t$ by $\tb$ and $G$ by $\CCC$.
\enp

\section{Proof of part 3 of Theorem \ref{fclt} -- Functional CLT}
\label{tight_sec}
We henceforth tacitly assume that $\r_d^{1/d}\to\ff$ and that the additive functional $\aaa$ satisfies the {conditions that $\aaa = \aaa^+ - \aaa^-$ where $\aaa^+$ and $\aaa^-$ are increasing nonnegative functionals such that $\max(\aaa^+(G),\aaa^-(G)) \in e^{O(|G|)}$}.
After having shown multivariate normality in Section \ref{mult_sec}, proving the functional CLT reduces to establishing tightness of $\big(\Var[A_{d,1}]^{-1/2}({A_{d, \cdot} - \E[A_{d, \cdot}]})\big)_{d \ge 1}$. The main work lies in proving the tightness. Since the sum of tight processes is tight, we may assume that $\aaa$ is increasing.

To prove tightness, we will verify a Chentsov-type moment condition. We will show that there exist
$C_*, \e > 0$ such that
\begin{align}
    \label{lav_eq}
    \r_d^{-2}\E\bigl[\,\overline{A}_{d, E}^4\bigr] \le C_*|E|^{1 + \e}
\end{align}
for all $d\ge 1$  and all intervals $E = [t_-, t_+] \su [0, 1]$, where 
$$\overline{A}_{d, E} := A_{d, E} - \E[A_{d, E}] := ( A_{d, t_+} - A_{d, t_-} ) - \E[A_{d, t_+} - A_{d, t_-}]$$
denotes the centered increment in the interval $E$. By \cite[Theorems 2 and 3]{bickel} or \cite[Theorem 2]{lavancier} this inequality implies tightness.

%
%BLUEPRINT
%
To achieve this goal, we build on the strategy that was already successfully implemented in \cite{cyl} for persistent Betti numbers of networks. First, we show that it suffices to verify the Chentsov condition for intervals $E$ that are \emph{$d$-big} in the sense that $|E|\ge \r_d^{-2/3}$. Then, we derive bounds on the variance and the fourth-order cumulant $c^4$ to deduce the result. 

%
%GRID LEM
%
\bepr[Reduction to $d$-big intervals]
    \label{grid_prop}
    If the Chentsov condition \eqref{lav_eq} holds for all $d \ge 1$ and $d$-big intervals $E \su \ot$, then the processes $\{\r^{-1/2}_d\overline{A}_{d,\,\cdot\,}\}_{d \ge 1}$ are tight in the Skorokhod topology.
\enpr
\bep
    %DAVYDOV
    The uniform expectation bounds in Theorem \ref{fclt}, part 1 give that 
    $\E[ A_{ d,E}]\in o(\r_d^{1/2})$ if $|E|\le \r_d^{-2/3}$. Hence, 
    \begin{align}
        \label{eq:dav}
        \sup_{E \su \ot \text{ $d$-small}}\r_d^{-1/2}\E[ A_{ d,E} ]\to 0.
    \end{align}
    Therefore, we satisfy all conditions of \cite[Theorem 2]{davydov}, namely $A_{ d, t}$ is increasing in $t$, \eqref{eq:dav} holds, the CLT for the finite dimensional marginal holds, and the Chentsov condition \eqref{lav_eq} holds for all $d \ge 1$ and $d$-big intervals $E \su \ot $. 
\enp
%
%VAR BOUND
%
\bepr[Cumulant bounds]
    \label{var_prop}
    It holds that
    $\sup_{d \ge 2}\sup_{\substack{E \su [0, 1] }} \r_d^{-1} c^4[A_{d, E}]< \ff.$
\enpr

Before establishing Proposition \ref{var_prop}, we explain how to derive Theorem \ref{fclt}.
%
%FCLT PRF
%
\bep[Proof of Theorem \ref{fclt}, part 3]
    First, observe that by applying the Mecke formula as in the proof of Lemma~\ref{lem:Meckeapplied}, one can show that the random variable $Z = A_{d, E}$ has finite fourth moments. Hence, it satisfies the cumulant identity $\E[(Z - \E[Z])^4] = 3\Var[Z]^2 + c^4[Z]$. Now, note that the uniform variance bounds in Theorem \ref{fclt}, part~1 imply that 
    $\sup_{d \ge 2}\sup_{\substack{E \su [0, 1] }} |E|^{-1}\r_d^{-1}\Var[A_{d, E}] < \ff$.
    Combining this with Proposition \ref{var_prop} shows that for a suitable $C > 0$, when $E$ is $d$-big, we have
    \begin{align*}
        \r_d^{-2}\E\bigl[\,\overline{A}_{d, E}^4\bigr] &= 3\r_d^{-2}\Var[A_{d, E}]^2 + \r_d^{-2}c^4[A_{d, E}] \le 3C^2 |E|^2 + C \r_d^{-1} \le 3C^2 |E|^{3/2} + C |E|^{3/2}.
    \end{align*}
    Hence, invoking Proposition \ref{grid_prop} concludes the proof.
\enp

%MOM MEAS
To establish Proposition \ref{var_prop}, we proceed along the lines of \cite{svane} and rely on the concept of cumulant and semi-cluster measures \cite{barysh,raic3}. To make the presentation less technical, we show only that 
\begin{align}
    \label{eq:74}
    \sup_{d \ge 2}\sup_{t \le 1}\r_d^{-1}c^4[A_{d, t}] < \ff.
\end{align}
Along similar lines one can show that $\sup_{d \ge 2}\sup_{t_1, \dots, t_4 \le 1}\r_d^{-1}c^4[A_{d, t_1},\dots, A_{d, t_4}] < \ff$ so that
the multilinearity of the mixed cumulant $c^4[\cdot, \cdot, \cdot, \cdot]$ implies the claim for the case of $A_{d,E}$.
First, recalling that $s_d(t) := t^{1/d}$, we write $ \Rips(x, \PP_d; s_d(t))$,  for the connected component of the Gilbert graph on $\PP_d \cup \{x\}$ at level $t^{1/d}$ if $x$ is the center of this component chosen according to a measurable and translation-covariant center function with values in $\PP_d \cup \{x\}$. Otherwise, we formally put $\Rips(x, \PP_d; s_d(t)) := \es$. Now, we define
\begin{align*}
    \mu_d := \sum_{{x \in \PP_d }} \aaa(\Rips(x, \PP_d; s_d(t))) \delta_{x} ,
\end{align*}
where $\delta_x$ denotes the Dirac measure.
Next, writing $\lan f, \mu \ran := \int f(x)\mu(\!\d x)$, we introduce its $n$th \textit{moment measure} and the $n$th \textit{cumulant measure} through
$$\lan f, M_d^n\ran := \E[\lan f_1, \mu_d\ran \cdots \lan f_n, \mu_d\ran], \quad \text{ and }\quad \lan f, c_d^n\ran := c^n[\lan f_1, \mu_d\ran,\dots, \lan f_n, \mu_d\ran],$$
where $f(x_1,\dots, x_n )= f_1(x_1) \cdots f_n(x_n)$  with $f_1, \dots, f_n \co W_d \to \R$ bounded and measurable.

Next, we decompose the cube $W_d^k$ according to how best to partition a configuration $\xx\in W_d^k$ of $k$ points in two clusters so that the distance between these clusters is maximized.
For a finite set $U$, we set $\xx_U := (x_i)_{i \in U} \in W_d^{|U|}$.
Then, we write
\begin{align*}
    D(\xx_U) &:= \max_{\{S, T\} \preceq U}\dist(\xx_S, \xx_T)
    =\max_{S,T} \min_{i \in S, j \in T}|x_i- x_j| ,
\end{align*}
where $\{S, T\} \preceq U$ is a partition with $S, T \ne \es$.
Note, that this can also be interpreted as the longest edge-length in a spanning tree of $\xx_U$ which maximizes this quantity. This interpretation implies the following $l_\infty$-diameter bounds
\begin{equation}
    \label{diam_bound}
    D(\xx_U) \leq \diam(\xx_U) \leq (|U| - 1) D(\xx_U) .
\end{equation}
Then, for any $\{S,T\} \preceq U$, the set $\s(S, T) := \{\xx\in W_d^{|U|}\co D(\xx) = \dist(\xx_S, \xx_T) > 0\}$ 
describes the family of configurations whose best separated clusters are $\xx_{S}$ and $\xx_{T}$.
The configurations which do not belong to any of these families belong to the diagonal $\De_d^{|U|} := \{(x, \dots, x)\co x\in W_d \} \su W_d^{|U|}$.
As in \cite[Equation (3.28)]{raic3}, we decompose the cumulant measure into a diagonal and an off-diagonal contribution. That is 
\begin{align}
    \label{dod_eq}
    c^4[A_{d,t}] = \lan 1,  c_d^4 \ran = \int_{\De_d^4 }1\d c_d^4 + \sum_{S, T}\int_{\s(S, T)} 1 \d c_d^4,
\end{align}
where the sum ranges over all unordered non-trivial partitions of \{1,2,3,4\} into two sets.

To prove Proposition \ref{var_prop}, we need suitable moment bounds for $\aaa\big(\Rips(o, \PP_d; s_d(t))\big) \le \aaa\big(\Rips(o, \PP_d; 1)\big)$.
%
%OMOM LEM
%
\bel[Moment bound]
    \label{omom_lem}
    Let $n \ge 1$. Then, $\sup_{\xx \in W_d^4}\E[\aaa\big(\Rips(o, \PP_d\cup \xx; 1)\big)^n] \in O\big(\r_d\la_d^{-d}|W_d|^{-1}\big)$.
\enl
%
%OMOM PRF
%
\bep[Proof of Lemma \ref{omom_lem}]
    Let $\aaa^{(n)}$ be the additive functional defined by $\aaa^{(n)}(G) = \aaa(G)^n$ for any connected graph~$G$. Set $A_{d,1}^{(n)} := \aaa^{(n)}(\Rips(\PP_d; 1))$.
    We have $\E [A_{d,1}^{(n)}] = \la_d^d |W_d| \E[ \aaa \big(\Rips(o, \PP_d; 1)\big)^n] $ as a consequence of the Mecke formula.
    On the other hand, part 1 of Theorem \ref{fclt} implies $ \E [A_{d,1}^{(n)}]
    \in O(\r_d)$, which yields the lemma.
\enp

%We prove Lemma \ref{omom_lem} at the end of this section. 
Now, for a subset $S \su \{1, \dots, n\}$, we extend the definition of the \textit{moment measure} by setting 
$\lan f, M_d^S\ran := \E\bigl[\prod_{s \in S}\lan f_s, \mu_d\ran \bigr]$, where $f(\xx_S)= \prod_{s \in S} f_s(x_s)$  with $f_s \co W_d \to \R$ bounded and measurable. 
We also work frequently with the \emph{mixed moments}
\begin{align}
    \label{mxmom_eq} 
    m_d^{S}(\xx_S) := \E\Bigl[\prod_{i \in S} \aaa \big(\Rips(x_{i}, \PP_d \cup \xx_S; s_d(t))\big) \Bigr],
\end{align}
defined for any $\xx_S \in W_d^{|S|}$.

To deduce Proposition \ref{var_prop}, we proceed as in \cite[Section 3.1]{raic3} and leverage that the moment measures can be expressed as 
\begin{align}
    \label{mxd_rep_eq}
    \d M_d^S = \sum_{\{S_1, \dots, S_p\}\preceq S} \la_d^{d p} m_d^{S}\d_{\ms s} \xx_{S_1} \cdots \d_{\ms s} \xx_{S_p},
\end{align}
where the $\d_{\ms s} \xx_{S_i}$ are the singular measures, i.e.,
\begin{equation}
    \label{sing_meas}
    \int_{\R^{d|S|}}g(\xx_S) { \d_{\ms s} \xx_{S} } %\d_{\ms s} \xx_{S_1} \cdots \d_{\ms s} \xx_{S_p} 
    = \int_{\R^d}g(x, \dots, x) \d x .
\end{equation}

%
%DIAG PROOF
%
\bel[Diagonal contribution]
    \label{lem:diag}
    It holds that $\int_{\De_d^4 }1\d c_d^4 =\la_d^d \, |W_d| \, \E\bigl[ \aaa \big(\Rips(o, \PP_d; s_d(t))\big)^4\bigr].$
\enl
\bep
    Note that arguing as in \cite[Lemma 3.1]{raic3}, we have that
    $\int_{\De_d^4 }1\d c_d^4 = \la_d^d\int_{W_d} \E\bigl[ \aaa \big(\Rips(x, \PP_d; s_d(t))\big)^4\bigr] \d x
    = \la_d^d \, |W_d| \, \E\bigl[ \aaa \big(\Rips(o, \PP_d; s_d(t))\big)^4\bigr]$.
\enp

To deal with the off-diagonal contributions, we fix nonempty disjoint sets $S, T \su \{1, 2, 3, 4\}$ with $S \cup T = \{1, 2, 3, 4\}$. Then, we will leverage the semi-cluster decomposition of the cumulant measure \cite{barysh,raic3}. To render the presentation self-contained, we recollect here some of the most fundamental properties, referring the reader to the aforementioned sources, for further details. %DIAG
More precisely, writing $C_d^{S_1, T_1} := M_d^{S_1 \cup T_1} - M_d^{S_1}M_{d}^{T_1}$ and $M_d^{\mc S_2} := \prod_{S' \in \mc S_2} M_d^{S'}$, we represent $c_d^4$ as
\begin{align}
    \label{st_dec}
    c_d^4 = \sum_{\substack{ \es \ne S_1 \su S\\ \es \ne T_1 \su T}}\sum_{ (\mc S_2, \mc T_2)}\a_{(S_1, T_1), (\mc S_2, \mc T_2)}C_d^{S_1, T_1} M_d^{\mc S_2}M_d^{\mc T_2},
\end{align}
for some coefficients $\a_{(S_1, T_1), (S_2, T_2)} \in \R$, 
where the second sum extends over all partitions $\mc S_2$ of $S \sm S_1$ and $\mc T_2$ of $T \sm T_1$.
Henceforth, we fix such choices $(S_1, T_1)$, $(\mc S_2, \mc T_2)$. 

In order to bound integrals over the signed measure $C_d^{S_1, T_1} $, we decompose the integration domain accordingly to the distance to the diagonal. More precisely, for $\ell \ge 1$, we set
$$U_{d, \ell} := \big\{ \xx_{S_1 \cup T_1} \in W_d^{S_1 \cup T_1} \co \ell -1 < D(\xx_{S_1 \cup T_1}) \le \ell\big \}.$$
Note, that for any $\xx_{S_1 \cup T_1} \in U_{d,\ell}$, by \eqref{diam_bound} we have $ \ell - 1 \le \diam(\xx_{S_1 \cup T_1}) \le ( |S_1 \cup T_1| - 1 ) \ell$, from which it follows that
\begin{equation} 
    \label{bndUdl}
    U_{d,\ell}
    \subseteq (\De_d^{S_1 \cup T_1} + B_{d |S_1 \cup T_1|} (o,3\ell) )\setminus \De_d^{S_1 \cup T_1} .
\end{equation}
A key step in the proof is to bound the volume of $U_{d, \ell}$.
Note that to simplify the expressions in further arguments, we choose not to optimize certain upper bounds.

%
%VOL LEM
%
\bel[Volume bound on $U_{d, \ell}$]
    \label{uvol_lem}
    Let $d, \ell \ge 1$. Then, $|U_{d, \ell}| \le c_{\ms A}^d\ell^{3d}|W_d|$ for some $c_{\ms A}> 1$. 
\enl
\bep
    By \eqref{diam_bound}, on the set $U_{d, \ell}$ the $l_\ff$-diameter of $\xx_{S_1 \cup T_1}$ is at most ${(|S_1\cup T_1|-1) D(\xx_{S_1\cup T_1}) \leq } 3\ell$.
    Hence, by the Fubini theorem,
    $|U_{d, \ell}| \le |W_d| \, (3\ell)^{d (|S_1 \cup T_1| -1)}\le 27^d \ell^{3d} |W_d|$.
\enp
The remainder of the proof is based on two auxiliary results whose proofs will be deferred to the end of this section. First, we bound the integral over the moment measures $M_d^{\mc S_2}$ and $M_d^{\mc T_2}$.

%
%MOMM-LEM
%
\bel[Bound for the moment measures]
    \label{mk_bound_lem}
    There exists constants $\cmm, d_{\ms{MM}} \ge 1$ with the following property. Let $r \ge1$, and $\mc S'$ be a partition of a nonempty set $S' \su \{1, 2, 3, 4\}$. Then, for all $d \ge d_{\ms{MM}}$,
    \been
        \im it holds that $M_d^{\mc S'}(B_{d|S'|}(y, r)) \le \cmm^d r^{d |S'|} \r_d|W_d|^{-1}$ for every $y \in W_d^{ |S'|}$;
        \im it holds that $M_d^{\mc S'}\big(B_{d|S'|}(y, r)\sm\De_d^{|S'|}\big) \le \cmm^d r^{d |S'|} \la_d^d\r_d|W_d|^{-1}$ for every $y \in W_d^{ |S'|}$;
        \im it holds that $M_d^{ S'}\big(\De_d^{|S'|} + B_{d|S'|}(o, r)\big) \le \cmm^d r^{d |S'|} \r_d$;
        \im it holds that $M_d^{ S'}\big((\De_d^{|S'|} + B_{d|S'|}(o, r)) \sm \De_d^{|S'|}\big) \le \cmm^d r^{d |S'|} \la_d^d \r_d$.
    \enen
\enl

The final step is to bound the signed measure on $C_d^{S_1, T_1} $.

%
%COR DEC
%
\bel[Fast decay of correlations]
    \label{cor_dec_lem}
    There exists a constant $c_{\ms{DC}} > 0$ with the following property.
     Assume that $\xx_{ S_1 \cup T_1}$ is such that $\ell := \dist(\xx_{S_1}, \xx_{ T_1}) \ge k_{0,+}:= 64(k_0 + 1)$.
    Then
    \begin{align*}
        \big|m_d^{ S_1 \cup  T_1}(\xx_{ S_1 \cup T_1}) - m_d^{ S_1}(\xx_{ S_1}) m_d^{ T_1}(\xx_{ T_1}) \big| \le c_{\ms{DC}}^{\ell d} \la_d^{\ell d/32}.
    \end{align*}
\enl

We now elucidate how to formally conclude the proof of Proposition \ref{var_prop}. After that, the main work lies in establishing Lemmas \ref{mk_bound_lem} and \ref{cor_dec_lem}.

%
%ODL PRF
%
\bep[Proof of \eqref{eq:74}]
    First, Lemma \ref{lem:diag} together with Lemma \ref{omom_lem} deal with the diagonal contributions in \eqref{dod_eq}.
    Therefore, by the decompositions \eqref{dod_eq} and \eqref{st_dec} of $c_d^4$, it remains only to bound by a constant, independent of $d$ and $t$, the quantities $\r_d^{-1} C_d^{S_1, T_1} M_d^{\mc S_2}M_d^{\mc T_2} (\s(S, T) )$, where $\{ S,T\}$ is a nontrivial partition of $\{1,2,3,4\}$, $S_1$ and $T_1$ are nonempty subsets of $S$ and $T$, and $\mc S_2$ and $\mc T_2$ are partitions of $S\setminus S_1$ and $T\setminus T_1$.
    
    Let $S,T,S_1,T_1, \mc S_2, \mc T_2$ be as above.
    Let $\xx\in\s(S,T)$.
    Then, we observe
    \[ D(\xx_{S_1\cup T_1})
    \le \diam(\xx_{S_1\cup T_1})
    \le \diam(\xx_{S\cup T})
    \le 3 D(\xx_{S\cup T})
    = 3 \dist(\xx_S,\xx_T)
    \le 3 \dist(\xx_{S_1},\xx_{T_1})
    \le 3 D(\xx_{S_1 \cup T_1}),\]
    where, the first and third inequalities are given by \eqref{diam_bound}, the second and fourth inequalities follows from the inclusions $S_1\subseteq S$ and $T_1\subseteq T$, the last inequality is implied by the definition of $D$, and the unique equality is due to $\xx\in\s(S,T)$.
    Note also that both $\diam(\xx_S)$ and $\diam(\xx_T)$ are bounded by $\diam(\xx_{S\cup T})$.
    Therefore, for any $\xx_{S_1 \cup T_1} \in U_{d,\ell} \cap \s(S_1,T_1) $, since $\ell-1 \le D(\xx_{S_1\cup T_1}) \le \ell$, one has
    \[ \dist(\xx_{S_1}, \xx_{T_1}) \ge (\ell-1)/3 ,
    \quad \xx_{S \sm S_1} \in B_{d(|S|-|S_1|)}(x_{\min(S_1)}, 3 \ell) ,
    \quad \text{and} \quad \xx_{T \sm T_1} \in B_{d(|T|-|T_1|)}(x_{\min(T_1)}, 3 \ell) .\]
    Hence, applying part 1 of Lemma \ref{mk_bound_lem} twice yields that 
    \begin{align*}
        C_d^{S_1, T_1} M_d^{\mc S_2} M_d^{\mc T_2} (\s(S, T))
        &\leq \int_{\s(S, T) } \big|C_d^{S_1, T_1}\big|M_d^{\mc S_2}M_d^{\mc T_2}(\!\d \xx)
        \\
	    &\le \sum_{\ell \ge 1} \int_{U_{d, \ell}'} \int_{B_{d(|S|-|S_1|)}(x_{\min(S_1)}, 3 \ell)} \int_{B_{d(|T|-|T_1|)}(x_{\min(T_1)}, 3 \ell)} \hspace{-2.8cm} M_d^{\mc T_2} (\!\d\xx_{ T \sm T_1}) \, M_d^{\mc S_2}(\!\d\xx_{S \sm S_1}) \, \big|C_d^{S_1, T_1}\big|(\!\d \xx_{S_1 \cup T_1}) \\
	    &\le \cmm^{2d} \r_d^{2}|W_d|^{-2} \sum_{\ell \ge 1}\int_{U_{d, \ell}'} (3 \ell)^{4d}\big|C_d^{S_1, T_1}\big|(\!\dxx_{S_1 \cup T_1}) .%,
    \end{align*}
    where 
    $U_{d, \ell}':= U_{d, \ell} \cap \big\{\xx_{S_1 \cup T_1} \in W_d^{S_1 \cup T_1} \co\dist( \xx_{S_1}, \xx_{T_1})> (\ell-1)/3\big \}$.
    Now,
    \begin{align}
        \notag
        \sum_{\ell \ge 1}\int_{U_{d, \ell}'} (3\ell)^{4d}\big|C_d^{S_1, T_1}\big|(\!\dxx_{S_1 \cup T_1})
        &\le \sum_{\ell \le 3k_{0,+}} (3\ell)^{4d} \int_{ U_{d, \ell}'} \big(M_d^{S_1 \cup T_1} + M_d^{S_1}M_d^{T_1}\big)(\!\dxx_{S_1 \cup T_1})\\
        &\phantom= +\sum_{\ell \ge 3k_{0,+} + 1} (3\ell)^{4d} \int_{U_{d, \ell}'} \big|C_d^{S_1, T_1}\big|(\!\dxx_{S_1 \cup T_1}).
        \label{secondsum}
    \end{align}
    We deal with the two sums separately.  We start by the first, which involves a finite number of summands, so we only need to show that each of these summands is of order at most $\r_d$.
    First, {as an immediate consequence of \eqref{bndUdl}, we} note that $U_{d, \ell}'$ is disjoint from the diagonal for $\ell \ge 1$. Hence, applying part 4 of Lemma \ref{mk_bound_lem} with $S' = S_1 \cup T_1$ and $r =  {3\ell}$ gives that 
    $$ (3 \ell)^{4d} \int_{U_{d, \ell}'} M_d^{S_1 \cup T_1} (\!\dxx_{S_1 \cup T_1}) 
    \le (3 \ell)^{4d} \cmm^d (3\ell)^d\r_d \la_d^d
    \le \r_d ,$$
    where the last inequality holds for $\ell\leq k_{0,+}$ and $d$ large enough.
    As a consequence of \eqref{bndUdl}, we have
    \begin{align*}
        \int_{ U_{d, \ell}'} M_d^{S_1}M_d^{T_1} (\!\dxx_{S_1 \cup T_1})
        &\leq \int_{(\De_d^{|T_1|} + B_{d|T_1|}(o, 3\ell)) \sm \De_d^{|T_1|}} \int_{B_{d|S_1|}(x_{\min{T_1}}, 6 \ell)) \sm \De_d^{|S_1|}} M_d^{S_1} (\!\dxx_{S_1}) M_d^{T_1} (\!\dxx_{T_1}).
    \end{align*} 
    Thus, applying parts 2 and 4 of Lemma \ref{mk_bound_lem}, and using that $|S_1|+|T_1|\le 4$, shows that 
    $$
    \ell^{4d} \int_{U_{d, \ell}'} M_d^{S_1}(\!\dxx_{S_1})M_d^{T_1} (\!\dxx_{T_1}) 
    \le\ell^{4d} \cmm^{2d} (6 \ell)^{4d} \r_d^2 \la_d^{2d} |W_d|^{-1}
    \le c^d \r_d \la_d^{d(k_0+3)}
    \le \r_d,
    $$
    for an appropriate constant $c>0$, where the two last inequalities hold for $\ell\le k_{0,+}$, and $d$ big enough.
    
    It remains only to show that the sum \eqref{secondsum} is of order at most $\r_d$.
    Let $\ell \ge 3k_{0,+} + 1$. 
    Since $\dist(\xx_{S_1},\xx_{T_1}) > (\ell-1)/3 \geq 0$ for any $\xx_{S_1\cup T_1}\in U_{d,\ell}'$, one has that $\d_{\ms s} \xx_{U_1} \cdots \d_{\ms s} \xx_{U_p}$ vanishes on $U_{d,\ell}'$ for any partition $\{U_1,\ldots,U_p\} \preceq S_1\cup T_1$ which is not the union of partitions $\{S_1', \dots, S_{p_S}'\}\preceq S_1$ and $\{T_1', \dots, T_{p_T}'\}\preceq T_1$. 
    Therefore, from the definition $C_d^{S_1,T_1} = M_d^{S_1 \cup T_1}- M_d^{S_1} M_d^{T_1} $ and \eqref{mxd_rep_eq}, we get
    $$ \d C_d^{S_1,T_1} 
    = \sum_{\{S_1', \dots, S_p'\}\preceq S_1} \sum_{\{T_1', \dots, T_q'\}\preceq T_1} \la_d^{d (p + q)} (m_d^{S_1 \cup T_1} - m_d^{S_1} m_d^{T_1} )\d_{\ms s} \xx_{S_1'} \cdots \d_{\ms s} \xx_{S_p'}\d_{\ms s} \xx_{T_1'} \cdots \d_{\ms s} \xx_{T_q'}
    , \quad \text{on }  U_{d,\ell}'.
    $$
    Note also, that by  Lemma~\ref{cor_dec_lem}, we have
    \[ |m_d^{S_1 \cup T_1} - m_d^{S_1} m_d^{T_1}|
    \le c_{\ms{DC}}^{(\ell-1) d / 3} \la_d^{(\ell-1) d/3\times 32} 
    \le c_{\ms{DC}}^{\ell d} \la_d^{\ell d/100} 
    \quad \text{ on $U'_{d,\ell}$ when $(\ell-1)/3 \ge k_{0,+} $,}\]
    where the second inequality holds because $(\ell-1) /3\times 32 \ge \ell/100$ when $\ell \ge 25$, which is ensured by $(\ell-1)/3 \ge k_{0,+} \geq 64$.
    Hence, we deduce from Lemma~\ref{uvol_lem} and the inequality $\ell \le e^\ell$ that, for sufficiently large $d$,
    \begin{align*}
        \sum_{\ell \ge 3k_{0,+} + 1} (3 \ell)^{4d}\int_{U_{d, \ell}'} \big|C_d^{S_1, T_1}\big|(\!\dxx_{S_1 \cup T_1}) 
        &\le \sum_{\ell \ge 3k_{0,+} + 1} (3 \ell)^{4d} c_{\ms{DC}}^{\ell d} \la_d^{\ell d/100} c_{\ms A}^d\ell^{3d}|W_d|
        \\&\le 3^{4d} c_{\ms A}^d |W_d| \sum_{\ell \ge 3k_{0,+} + 1} (e^{7d} c_{\ms{DC}}^{d} \la_d^{d/100} )^{\ell}
        \\&\le 3^{4d} c_{\ms A}^d |W_d| (e^{7d} c_{\ms{DC}}^{d} \la_d^{d/100} )^{3k_{0,+}}
        % \\&= \big(3^4 c_{\ms A} (e^{7} c_{\ms{DC}})^{3 k_{0,+}} v_{\max}^{-1} \la_d^{\frac{3 k_{0,+}}{100} - (k_0+1)} \big)^d  \r_d.
        \\&= \big(c \, \la_d^{\frac{3 k_{0,+}}{100} - (k_0+1)} \big)^d  \r_d ,
    \end{align*}
    for an appropriate constant $c>0$.
    Recalling that $k_{0,+} = 64(k_0+1)$, we see that the right-hand side is of order $o(\r_d)$. It concludes the proof.
\enp

We conclude this section by proving Lemmas %\ref{omom_lem}--\ref{cor_dec_lem}
\ref{mk_bound_lem} and \ref{cor_dec_lem}.

%
%PRF MKBLEM
%
\bep[Proof of Lemma \ref{mk_bound_lem}]
    We first establish a general bound that will be used to proved the four parts of the lemma.
    Let $B \subseteq \R^{d|S'|}$ be an arbitrary measurable set.
    From expressions \eqref{mxd_rep_eq} and \eqref{mxmom_eq} describing mixed moment measures we have
    $$ M_d^{S'}(B)
    = \sum_{p \le |S'|}\la_d^{d p}\sum_{\{S_1, \dots, S_p\}\preceq S'}\int_{B} \E\Bigl[\prod_{i \le p} \aaa\big(\Rips(x_{\min(S_i)}, \PP_d \cup \xx_{S'}; s_d(t))\big)^{|S_i|}\Bigr] \d_{\ms s} \xx_{S_1} \cdots \d_{\ms s} \xx_{S_p} .$$
     
    Note that if $X_1, \dots, X_k$ are nonnegative random variables, their mixed moments can be bounded in terms of their moments as follow. For any $n_1, \dots, n_k \ge 1$, it holds that
    $\E\big[\prod_{i \le k} X_i^{n_i}\big] \le \prod_{i \le k} \E[X_i^{m}]^{n_i/m},$
    where $m = n_1 +\cdots + n_k$. For $k=2$, this is Hölder's inequality. The general case follows by iteration.
    Therefore, using that $\aaa$ is increasing and invoking moment bounds from Lemma \ref{omom_lem} gives that 
    \begin{align}
        \label{MdSboundgen}
        M_d^{S'}(B)
        &\le c \r_d \la_d^{-d}|W_d|^{-1} \sum_{p \le |S'|}\la_d^{d p}\sum_{\{S_1, \dots, S_p\}\preceq S'}\int_{B} 1 \d_{\ms s} \xx_{S_1} \cdots \d_{\ms s} \xx_{S_p} ,
    \end{align}
    for a suitable $c > 0$.
    
    {\bf Part 1 \& 2.}
    We first prove the bound for a single-set partition, i.e, $\mc{S}'=\{S'\}$, which implies $M_d^{\mc{S}'}=M_d^{S'} $.
    From the definition \eqref{sing_meas} of the singular measures 
    $\d_{\ms s} \xx_{S_1},\ldots, \d_{\ms s} \xx_{S_p}$, we have
    \begin{equation*}
        \int_{B_{d|S'|}(y, r)} 1 \d_{\ms s} \xx_{S_1} \cdots \d_{\ms s} \xx_{S_p}
        = (2 r)^p ,
        \quad \text{ and } \quad 
        \int_{B_{d|S'|}(y, r) \cap \De_d^{|S'|}} 1 \d_{\ms s} \xx_{S_1} \cdots \d_{\ms s} \xx_{S_p}
        = 2 r \one \{ p = 1 \},
    \end{equation*}
    for any partition $\{S_1, \dots, S_p\}\preceq S'$.
    Note that for $p=1$, these two integrals are equal.
    Therefore, \eqref{MdSboundgen} gives
    \begin{align}
        \label{boundpart1singleset}
        M_d^{S'}(B_{d|S'|}(y, r))
        &\le c \r_d \la_d^{-d}|W_d|^{-1} \sum_{p \le |S'|}\la_d^{d p}\sum_{\{S_1, \dots, S_p\}\preceq S'} (2 r)^{d p}
        \le c' \r_d |W_d|^{-1} (2r)^{d |S'|} ,
        \intertext{and}
        \notag
        M_d^{S'}(B_{d|S'|}(y, r) \sm \De_d^{|S'|})
        &\le c \r_d \la_d^{-d}|W_d|^{-1} \sum_{2 \le p \le |S'|}\la_d^{d p}\sum_{\{S_1, \dots, S_p\}\preceq S'} (2 r)^{d p}
        \le c' \r_d |W_d|^{-1}  \la_d^d (2r)^{d |S'|} ,
    \end{align}
    for a suitable $c' > 0$, thereby proving the asserted bounds on $M_d^{S'}(B_{d|S'|}(y, r))$ and $M_d^{S'}(B_{d|S'|}(y, r) \sm \De_d^{|S'|})$.
    
    For the case of an arbitrary partition $\mc{S}'$ of $S'$, we exploit that $M_d^{\mc S'} := \prod_{S'' \in \mc S'} M_d^{S''}$.
    Thus, the bounds \eqref{boundpart1singleset} for single-set partitions, applied to each element of $\mc{S}'$, provides
    \begin{equation}
        \label{boundpart1}
        M_d^{\mc S'}(B_{d|S'|}(y, r)) 
        \le \cmm^d r^{d|S'|} (\r_d|W_d|^{-1})^{|\mc S'|} .   
    \end{equation}
    This implies the claim of part 1, since  $\r_d|W_d|^{-1} \in o(1)$.
    For part 2, we can assume furthermore that $|\mc{S}'|\geq 2$ since the case $|\mc{S}'| = 1$ is proven above. Observing that $\r_d|W_d|^{-1} = \la_d^{d} (\la_d^{k_0} \vm)^d$, we get from \eqref{boundpart1} that
    $$ M_d^{S'}(B_{d|S'|}(y, r) \sm \De_d^{|S'|})
    \le M_d^{\mc S'}(B_{d|S'|}(y, r)) 
    \le \cmm^d r^{d|S'|} \r_d|W_d|^{-1} \la_d^d ,$$
    as claimed.
    
    {\bf Part 3 \& 4.}
    In order to shorten notation, we set $B:= \De_d^{|S'|} + B_{d |S'|}(o, r)$.
    Note that for any partition $\{S_1, \dots, S_p\}\preceq S'$ and any $(\xx_{S_1}, \dots, \xx_{S_p}) \in B$ the $l_\ff$-distance of any {$\xx_{s_i}$ to $\xx_{s_1}$ is at most $2r$, for arbitrary $s_i\in S_i$ and $s_1\in S_1$}. Thus,
    \begin{align*}
        \int_B 1\d_{\ms s} \xx_{S_1} \cdots \d_{\ms s} \xx_{S_p} 
        &\le (2r)^{d (p - 1)}\int_{W_d^{S_1}} 1\d_{\ms s} \xx_{S_1}
        = (2r)^{d (p - 1)}|W_d|, 
        && \text{if } p\geq 1 , \text{ and }
        \\\int_B 1\d_{\ms s} \xx_{S_1} 
        &= \int_{\De_d^{|S'|}} 1\d_{\ms s} \xx_{S_1},
        && \text{if } p=1 .
    \end{align*}
    Thus \eqref{MdSboundgen} gives
    \begin{align*}
        M_d^{S'}(B)
        &\le c \r_d \sum_{p \le |S'|} (2r\la_d)^{d(p-1)} 
        \le c'' \r_d r^{d (|S'|-1)} ,
        \text{ and}\\
        M_d^{S'}(B \sm \De_d^{|S'|}) 
        &\le c \r_d \sum_{2 \le p \le |S'|} (2r\la_d)^{d(p-1)} 
        \le c'' \r_d (2 \la_d)^d r^{d (|S'|-1)} ,
    \end{align*}
    for a suitable $c''>0$, proving the claimed bounds.
\enp

To prove Lemma \ref{cor_dec_lem}, we need to bound the probability that the typical component in the Gilbert graph is large. We write $\mc L(\varphi) := \sup_{x\in \varphi}| x|$ for the maximal distance of an element of a locally finite $\vp \su \R^d$ to $o$ and let $ \Rips'(x, \PP_d)$ denote the connected component of the Gilbert graph on $\PP_d \cup \{x\}$ at level $1$ containing $x$.
\bel[Low probability for large components]
    \label{ark_lem}
    Let $d , \ell \ge 1$. Then,
    $\P\big(\mc L\big(\Rips'(o, \PP_d)\big) > \ell \big) \le (2\la_d)^{d \ell }.$
\enl
\bep
    If $\mc L\big(\Rips'(o, \PP_d)\big) > \ell$, then there exists a self-avoiding path in the Gilbert graph starting from the origin and consisting of at least $\ell$ further distinct nodes. By the Mecke formula, the probability for such a path to exist is at most the asserted $(2\la_d)^{d \ell}$.
\enp

\bep[Proof of Lemma \ref{cor_dec_lem}]
    %STAB RAD
    We define the \textit{stabilization radius}
    $R (\xx_{S_1}) := \max_{x \in \xx_{S_1}} \mc L(\Rips'(x, \PP_d \cup \xx_{S_1})- x)$
    to be the maximal elongation of the connected components centered at some $x \in \xx_{S_1}\in W_d^{S_1}$. The key property of such a stabilization radius is that it allows for a factorization of the mixed moments. More precisely, for a set  $S''\supseteq S_1$, we may decompose the mixed moment as 	
    $m_d^{S''} = m_{d, \le}^{S''} + m_{d, >}^{S''}$, where 
    \begin{align*}
        m_{d, \le}^{S''} (\xx_{S''}) &:= \E\Big[\one\{R (\xx_{S_1}) \le \ell\}\prod_{i \in S''} \aaa(\Rips(x_{i}, \PP_d \cup \xx_{S''}; s_d(t)))^{}\Big];\\
        m_{d, >}^{S''}( \xx_{S''}) &:= \E\Big[\one\{R (\xx_{S_1}) >\ell\}\prod_{i \in S''} \aaa(\Rips(x_{i}, \PP_d\cup \xx_{S''}; s_d(t)))^{}\Big].
    \end{align*}
    Now, the spatial independence of the Poisson point process $\PP_d$ implies that $m_{d, \le}^{S_1 \cup T_1} = m_{d, \le}^{S_1} m_d^{T_1}$. Thus,
    \begin{align*}
        |m_d^{S_1 \cup T_1} - m_d^{S_1}m_d^{T_1}| = |m_{d, >}^{S_1 \cup T_1} - m_{d, >}^{S_1}m_d^{T_1}| \le m_{d, >}^{S_1 \cup T_1} + c\la_d^d m_{d, >}^{S_1},
    \end{align*}
    where in the last step we invoked Lemma \ref{omom_lem} and that $\r_d|W_d|^{-1} \in O(\la_d^d)$ to see that $m_d^{T_1} \in O(\la_d^d)$. As above, we recall that if $X_1, \dots, X_k$ are nonnegative random variables and if $n_1, \dots, n_k \ge 1$, then 
    $\E\big[\prod_{i \le k} X_i^{n_i}\big] \le \prod_{i \le k} \E[X_i^{m}]^{n_i/m},$
    where $m = n_1 +\cdots + n_k$. Hence, by Lemma \ref{omom_lem} and Cauchy-Schwarz inequality,
    $ m_{d, >}^{S_1 \cup T_1} { (\xx_{S_1\cup T_1})}
    \le c' \sqrt{\P(R(\xx_{S_1}) > \ell )} $
    for a suitable $c' > 0$. Now, we note that $R(\xx_{S_1}) > \ell$ implies that $\mc L(\Rips'(x, \PP_d ) -x) > \ell/16$ for some $x \in \xx_{S_1}$. Thus, applying Lemma \ref{ark_lem}
    shows that $m_{d, >}^{S_1 \cup T_1} {(\xx_{S_1\cup T_1})} \le c'' (2\la_d)^{d \ell /32}$ for some $c''>0$.
    Noting that an analogous bound holds for $m_{d, >}^{ S_1}$ concludes the proof.
\enp

\section*{Acknowledgements}
The authors thank D.~Yogeshwaran for pointing out the relation between Poisson approximation and CLT elucidated in  Remark \ref{rem:pclt}.
The authors acknowledge financial support of the CogniGron research center and the Ubbo Emmius Funds (University of Groningen).

%%%%%%%%%%%%%%%%%%%%%%%%%
\bibliography{lit}

\end{document}